\documentclass{amsart}

\usepackage{hyperref}
\hypersetup{
    colorlinks=true,
    linkcolor=blue,
    citecolor=blue,
    urlcolor=blue,
}

\makeatletter
\newcommand\org@maketitle{}
\newcommand\@authors{}
\let\org@maketitle\maketitle
\def\maketitle{%
	\let\@authors\authors
	\nxandlist{; }{ and }{; }\@authors
	\hypersetup{
		linktocpage=true,
		pdftitle={\@title},
                pdfauthor={\@authors},
                pdfsubject={\subjclassname. \@subjclass},
		pdfkeywords={\@keywords}
	}%
	\org@maketitle
}
\makeatother

\usepackage{amssymb}
\usepackage{rsfso}
\usepackage{mathtools}
\usepackage{microtype}
\usepackage[alphabetic,msc-links]{amsrefs}
\usepackage{enumitem}
\usepackage{amsfonts}

\usepackage{doi}
\renewcommand{\PrintDOI}[1]{\doi{#1}}

\numberwithin{equation}{section}

\newtheorem{theorem}{Theorem}[section]
\newtheorem{lemma}[theorem]{Lemma}
\newtheorem{corollary}[theorem]{Corollary}
\newtheorem{proposition}[theorem]{Proposition}
\theoremstyle{definition}
\newtheorem{definition}[theorem]{Definition}
\newtheorem{example}[theorem]{Example}

\theoremstyle{remark}
\newtheorem{remark}[theorem]{Remark}

\newcommand{\cA}{{\mathcal A}}

\newcommand{\cJ}{{\mathcal J}}
\newcommand{\cH}{{\mathcal H}}

\newcommand{\al}{\alpha}
\newcommand{\be}{\beta}
\newcommand{\g}{\gamma}
\newcommand{\de}{\delta}

\newcommand{\la}{\lambda}
\newcommand{\ka}{\kappa}

\newcommand{\R}{\mathbb{R}}
\newcommand{\vp}{\varphi}

\newcommand{\bg}{\mathbf{g}}
\newcommand{\bff}{\mathbf{f}}

\newcommand{\D}{\nabla}

\renewcommand{\div}{\operatorname{div}}
\newcommand{\loc}{\mathrm{loc}}

\newcommand{\mean}[1]{\langle #1\rangle}

\def\Xint#1{\mathchoice
  {\XXint\displaystyle\textstyle{#1}}%
  {\XXint\textstyle\scriptstyle{#1}}%
  {\XXint\scriptstyle\scriptscriptstyle{#1}}%
  {\XXint\scriptscriptstyle\scriptscriptstyle{#1}}%
  \!\int}
\def\XXint#1#2#3{{\setbox0=\hbox{$#1{#2#3}{\int}$}
    \vcenter{\hbox{$#2#3$}}\kern-.5\wd0}}

\def\dashint{\Xint-}

\DeclareRobustCommand{\rchi}{{\mathpalette\irchi\relax}}
\newcommand{\irchi}[2]{\raisebox{\depth}{$#1\chi$}}

\mathchardef\ordinarycolon\mathcode`\:
\mathcode`\:=\string"8000
\begingroup \catcode`\:=\active
  \gdef:{\mathrel{\mathop\ordinarycolon}}
\endgroup

\author{Seongmin Jeon}
\author{Stefano Vita}
\address[S. Jeon]{Department of Mathematics
\newline\indent
KTH Royal Institute of Technology
\newline\indent
SE-100 44, Stockholm, Sweden}
\email{seongmin@kth.se}
\address[S. Vita]{Dipartimento di Matematica Giuseppe Peano
\newline\indent
Universit\`a degli Studi di Torino
\newline\indent
Via Carlo Alberto 10, 10124, Torino, Italy}
\email{stefano.vita@unito.it}

\title[Higher order boundary Harnack principles in Dini type domains]{Higher order boundary Harnack principles in Dini type domains}


\allowdisplaybreaks

\begin{document}

\begin{abstract}
Aim of this paper is to provide higher order boundary Harnack principles \cite{DeSSav15} for elliptic equations in divergence form under Dini type regularity assumptions on boundaries, coefficients and forcing terms. As it was proven in \cite{TerTorVit22}, the ratio $v/u$ of two solutions vanishing on a common portion $\Gamma$ of a regular boundary solves a degenerate elliptic equation whose coefficients behave as $u^2$ at $\Gamma$. Hence, for any $k\ge1$ we provide $C^k$ estimates for solutions to the auxiliary degenerate equation under double Dini conditions, actually for general powers of the weight $a>-1$, and we imply $C^k$ estimates for the ratio $v/u$ under triple Dini conditions, as a corollary in the case $a=2$.
\end{abstract}

\maketitle
%\tableofcontents
%%%%%%%%%%%%%%%%%%%%%%%%%%%%%%%%%%%%%%

\section{Introduction and main results}

For a bounded domain $\Omega\subset\R^n$, $n\ge2$, we consider a second-order elliptic operator in divergence form 
\begin{align}\label{eq:div-operator}
Lu=\div(A\D u)=\sum_{i,j=1}^n\partial_i(a_{ij}\partial_ju).
\end{align}
Here, the variable coefficient matrix $A=(a_{ij})_{i,j=1}^n$ is symmetric $A=A^T$ and satisfies the following conditions for some constant $\Lambda\ge1$
\begin{align}
\label{eq:assump-coeffi}
\begin{cases}
    \text{the uniform ellipticity: }\,\,\,\frac1\Lambda|\xi|^2\le \mean{A(x)\xi,\xi}\le\Lambda|\xi|^2,\quad\xi\in\R^n,\,\,\,x\in\Omega, \\
    \text{the uniform boundedness: }\,\,\,|A(x)|\le \Lambda,\quad x\in\Omega.
\end{cases}
\end{align}
Let us consider two functions $u,v$ solving
\begin{equation}\label{BHconditions}
\begin{cases}
-Lv=f &\mathrm{in \ }\Omega\cap B_1,\\
-Lu=g &\mathrm{in \ }\Omega\cap B_1,\\
u>0 &\mathrm{in \ }\Omega\cap B_1,\\
u=v=0, \quad \partial_{\nu} u<0&\mathrm{on \ }\partial\Omega\cap B_1,
\end{cases}
\end{equation}
where $0\in\partial\Omega$ and $\nu$ stands for the unit outward normal vector to $\Omega$ on $\partial\Omega$.

Starting from \cite{Kem72,Dal77,Anc78,Wu78}, an important field of research concerns the study of the regularity properties of the ratio $v/u$.
From one side, boundedness or even H\"older continuity of the ratio $v/u$ up to the boundary $\partial\Omega$ is known in literature as \emph{boundary Harnack principle}, and the research goes in the direction of proving the latter properties by lowering as much as possible the regularity assumptions on boundaries and coefficients, just to name a few contributions \cite{CafFabMorSal81,JerKen82,FabGarMarSal88,BanBasBur91,BasBur91,BasBur94,Fer98,DeSSav20,DeSSav22a,DeSSav22b}, or even adding right hand sides \cite{AllSha19,TorLat21}.

From the other side, the \emph{higher order boundary Harnack principle} concerns the validity of Schauder estimates for the ratio; that is, uniform continuity of the higher derivatives of $v/u$, and was proved for the first time in \cite{DeSSav15} in the elliptic case and then extended to the parabolic case in \cite{BanGar16,Kuk22}, see also \cite{DeSSav16} for the case of slit domains. 

Recently, an alternative proof of the latter result was proposed first in \cite{TerTorVit22} and then independently also in \cite{Zha23}. The new argument relies on the following observation: the ratio $w=v/u$ of two solutions of \eqref{BHconditions} solves the degenerate elliptic equation
\begin{equation}\label{e:r}
-\div\left(u^2A\nabla w\right)=u (f-gw) \quad\mathrm{in \ }\Omega\cap B_1;
\end{equation}
that is, it is a weak solution in the weighted Sobolev space $H^1(\Omega\cap B_1,u^2dx)$. Hence, the higher order boundary Harnack principle is implied by Schauder estimates up to the boundary $\partial\Omega$ for solutions to the auxiliary degenerate equation above.

Both the higher order boundary Harnack principle and the Schauder estimates for degenerate equations are nowadays well understood requiring boundaries, forcing terms and coefficients to be in $C^{k,\alpha}$ H\"older spaces (for the second topic one may refer to \cite{FabKenSer82,SirTerVit21a,SirTerVit21b,TerTorVit22}). Then, the purpose of the present work is to obtain these results when the modulus of continuity of the given data is of Dini type and hence possibly weaker than any $\alpha$-H\"older one and even not homogeneous. Far from being sharp, our results are a first attempt in lowering the requirements on the given data in order to still guarantee $C^k$ estimates for solutions to degenerate equations as in \eqref{e:r} and consequently for ratios of solutions to \eqref{BHconditions}.

%\begin{definition}[{\color{blue}1st possibility }$j$-Dini functions]\label{dinidef}
%We say that a nonnegative function $\omega:[0,1]\to\R$ is a \emph{Dini function} ($1$-Dini) if 
%\begin{itemize}
%\item[$(D1)$] \emph{(modulus condition)} $\omega$ is continuous and nondecreasing, $\omega(0)=0$, $\omega(1)>0$;
%\item[$(D2.1)$] \emph{(integrability condition)} $\omega(r) r^{-1}$ is integrable; that is
%$$
%\int_0^1\frac{\omega(r)}r\,dr<\infty;
%$$
%\item[$(D3)$] \emph{($\alpha$-nonincreasing condition)} there exists $\alpha\in(0,1)$ such that $\omega(r)r^{-\alpha}$ is nonincreasing.
%\end{itemize}
%Let
%\begin{equation}\label{Jomega}
%\mathcal{J}_\omega(r):=\int_0^r\frac{\omega(s)}s\,ds.
%\end{equation}
%In general, for any integer $j\in\mathbb N\setminus\{0\}$, we say that $\omega$ is $j$-Dini if satisfies $(D1),(D3)$ and
%\begin{itemize}
%\item[$(D2.j)$] \emph{($j$ integrability condition)}  $\cJ_\omega$ is $(j-1)$-Dini.
%\end{itemize}
%Here $0$-Dini means that no integrability assumption is required. We define $\mathcal D_j$ the set of $j$-Dini functions. 
%\end{definition}

\begin{definition}[$j$-Dini functions]\label{dinidef}
We say that a nonnegative function $\omega:[0,1]\to\R$ is a \emph{Dini function} ($1$-Dini) if 
\begin{itemize}
\item[$(D1)$] \emph{(modulus condition)} $\omega$ is continuous and nondecreasing, $\omega(0)=0$, $\omega(r)>0$ for $r>0$;
\item[$(D2)$] \emph{(integrability condition)} $\omega(r) r^{-1}$ is integrable; that is
$$
\int_0^1\frac{\omega(r)}r\,dr<\infty;
$$
\item[$(D3)$] \emph{($\alpha$-nonincreasing condition)} there exists $\alpha\in(0,1)$ such that $\omega(r)r^{-\alpha}$ is nonincreasing.
\end{itemize}
For any $j\in\mathbb N\setminus\{0\}$, let
\begin{equation}\label{Jomega}
\mathcal{J}^0_\omega(r)=\omega(r),\qquad\mathcal{J}^j_\omega(r):=\int_0^r\frac{\mathcal J^{j-1}_\omega(s)}s\,ds.
\end{equation}
We say that $\omega$ is $j$-Dini if satisfies $(D1),(D3)$ and $\mathcal{J}^{j-1}_\omega$ satisfies $(D2)$; that is,  $\mathcal{J}^{j}_\omega(1)<\infty$. We define $\mathcal D_j$ the set of $j$-Dini functions. 
\end{definition}

\begin{remark} We would like to remark here that for any $\alpha\in(0,1]$ condition $(D3)$ is not restrictive. In fact, given any function $\sigma$ satisfying the conditions in the previous definition apart from $(D3)$ and any $\alpha\in(0,1]$, one can construct a function $\sigma_\alpha$ (see \eqref{sigmaA}) such that the desired properties hold true together with condition $(D3)$. Moreover, without loss of generality one can work with the new modulus since $\sigma(r)\leq\sigma_\alpha(r)$. Although this fact was already remarked in \cite{ApuNaz16,ApuNaz22,DonEscKim18,DonLeeKim20} and proved in some subcases, for seek of completeness in Proposition \ref{alphaincreasing} we will give a unified proof which works in the full range $\alpha\in(0,1]$ and fits with our notion of $j$-Dini moduli. In other words, if $\sigma$ satisfies $(D1)$ and $\cJ^{j-1}_\sigma$ satisfies $(D2)$ for a given $j\in\mathbb N\setminus\{0\}$, we show that $\sigma_\alpha$ is $j$-Dini if one chooses $\alpha\in(0,1)$.
\end{remark}

For an integrable function $f$ in $\Omega$, we define nonnegative functions $\eta_f$, $\omega_f:(0,1]\to\R$ by 
\begin{align*}
&\eta_f(r):=\sup_{x\in\Omega}\sup_{y,z\in\Omega\cap B_r(x)}|f(y)-f(z)|,\\
&\omega_f(r):=\sup_{x\in\Omega}\dashint_{\Omega\cap B_r(x)}|f-\mean{f}_{\Omega\cap B_r(x)}|,
\end{align*}
where $\mean{f}_{\Omega\cap B_r(x)}=\dashint_{\Omega\cap B_r(x)}f$.

\begin{definition}[Notions of $j$-Dini continuity]
    Let $j\in\mathbb N\setminus\{0\}$. We say that a function $f\in L^1(\Omega)$ is \emph{uniformly $j$-Dini continuous} in $\Omega$ if $\eta_f$ is a $j$-Dini function. On the other hand, we say that $f$ is of \emph{$j$-Dini mean oscillation} in $\Omega$ if $\omega_f$ is a $j$-Dini function.
\end{definition}

Clearly, one has $\omega_f(r)\le \eta_f(r)$ and the Dini mean oscillation condition is weaker than the uniform Dini continuity condition.

\begin{definition}[$C^{k,\omega}$ spaces]
Let $\Omega\subset \R^n$ be a bounded domain. For $k\in\mathbb N$ and a modulus of continuity $\omega$ (i.e. a nonnegative function satisfying the modulus condition $(D1)$), we say that a function $u:\Omega\to\R$ is in $C^{k,\omega}(\Omega)$ if $u\in C^k(\overline\Omega)$ and $$
[D^\beta u]_{C^{0,\omega}(\Omega)}:=\sup_{x,y\in\Omega, x\neq y}\frac{|D^\beta u(x)-D^\beta u(y)|}{\omega(|x-y|)}<\infty
$$
for any multiindex $\beta=(\be_1,\cdots,\be_n)\in\mathbb N^n$ with $|\beta|=\sum_{i=1}^n\beta_i=k$.
Then, the norm is defined as
$$
\|u\|_{C^{k,\omega}(\Omega)}=\sum_{|\beta|\leq k}\|D^\beta u\|_{L^\infty(\Omega)}+\sum_{|\beta|=k}[D^\beta u]_{C^{0,\omega}(\Omega)}.
$$
\end{definition}
Let us remark that if $u\in C^{k,\omega}(\Omega)$ with $\omega$ a $j$-Dini function for some $j\in\mathbb N\setminus\{0\}$, then it means that $u\in C^k(\overline\Omega)$ and derivatives $D^\beta u$ of order $|\beta|=k$ are uniformly $j$-Dini continuous in $\Omega$.

\begin{definition}[$C^{k, j-Dini}$ domains]\label{diniboundary}
Let $k,j\in\mathbb N\setminus\{0\}$. We say that a bounded domain $\Omega\subset\R^n$ is $C^{k, j-Dini}$, if for each point $x^0\in\partial\Omega$, there exists $r>0$ such that after a possible rotation of coordinate axes, one has $$
    B_r(x^0)\cap\Omega=B_r(x^0)\cap\{x_n>\g(x')\}
    $$
    where $x'=(x_1,\cdots,x_{n-1})$, for a certain $C^{k,\omega}$ function $\g:\R^{n-1}\to\R$ for some $j$-Dini function $\omega$.
\end{definition}

Thanks to the previous definitions, we are now in position to state our results. Since we are interested in local estimates, up to rotations and dilations, we can locally parametrize the domain where $u,v$ solve \eqref{BHconditions} as the epigraph of a function $\g$; that is
$$\Omega\cap B_1=\{x_n>\g(x')\}\cap B_1,\qquad \partial\Omega\cap B_1=\{x_n=\g(x')\}\cap B_1.$$
Moreover, $\nu$ stands for the unit outward normal vector to $\Omega$ on $\Gamma:=\{x_n=\g(x')\}$.

\begin{theorem}[Higher order boundary Harnack principle in Dini type domains]\label{thm:Dini-bdry-harnack}
Let $k\in\mathbb N$, $j\in\mathbb N\setminus\{0,1,2\}$, and $\omega$ be a $j$-Dini function. Let us consider two functions $u,v\in H^1(\Omega\cap B_1)$ solving \eqref{BHconditions} with $L$ as in \eqref{eq:div-operator} and $A$ symmetric and satisfying \eqref{eq:assump-coeffi}. Let us assume that $A,f,g\in C^{k,\omega}(\Omega\cap B_1)$ and $\gamma\in C^{k+1,\omega}(B'_1)$. Then, $w=v/u$ belongs to $C^{k+1,\sigma}_\loc(\overline{\Omega}\cap B_1)$ where $\sigma$ is $(j-3)$-Dini, and satisfies the following boundary condition
\begin{equation}\label{boundaryHOBH}
2\mean{\nabla u,\nu}\mean{A\nabla w,\nu}+f-gw=0\qquad\mathrm{on \ } \partial\Omega\cap B_1.
\end{equation}
Moreover, if $\|A\|_{C^{k,\omega}(\Omega\cap B_{1})}+\|\gamma\|_{C^{k+1,\omega}(B_1')}+\|g\|_{C^{k,\omega}(\Omega\cap B_{1})}\le L_1, \|u\|_{L^2(\Omega\cap B_{1})}\le L_2$ and $\inf_{\partial\Omega\cap B_{3/4}}|\partial_{\nu}u|\ge L_3>0$, then the following estimate holds true
\begin{equation*}\label{eq.daperfez}
\left\|\frac{v}{u}\right\|_{C^{k+1,\sigma}(\Omega\cap B_{1/2})}\le C\left(\left\|v\right\|_{L^2(\Omega\cap B_{1})}+\|f\|_{C^{k,\omega}(\Omega\cap B_{1})}\right)
\end{equation*}
for any $u,v\in H^1(\Omega\cap B_1)$ satisfying \eqref{BHconditions} and with a positive constant $C$ depending on $n$, $\Lambda$, $\omega$, $k$, $L_1$, $L_2$, $L_3$. Finally, if $u(\vec e_n/2)=1$ and $v>0$ in $\Omega\cap B_{1}$, then
\begin{equation*}
\left\|\frac{v}{u}\right\|_{C^{k+1,\sigma}(\Omega\cap B_{1/2})}\le C\left(\left|\frac{v}{u}(\vec e_n/2)\right|+\|f\|_{C^{k,\omega}(\Omega\cap B_{1})}\right)
\end{equation*}
with a positive constant $C$ depending only on $n$, $\Lambda$, $\omega$, $k$, $L_1$. 
\end{theorem}
%In the previous statement, $0$-Dini simply means a generic modulus of continuity without any integrability condition.
Notice that in the statement above $\sigma$ is a $0$-Dini function when $j=3$. As a convention, $0$-Dini means a modulus of continuity satisfying $(D1),(D3)$ and with no integrability assumptions.

Let us now consider two functions $u,v\in H^1(B_1)$ solving $Lu=Lv=0$ in $B_1$ and sharing their zero sets; that is, $Z(u)\subseteq Z(v)$ where $Z(u):=u^{-1}(\{0\})$. From now on, we will refer to solutions to $Lu=0$ as $L$-harmonic functions.

The study of local regularity of the ratio $w=v/u$ across $Z(u)$ is called \emph{boundary Harnack principle on nodal domains}, and was initiated in \cite{LogMal15,LogMal16} in case of harmonic functions sharing zero sets. In this situation, the ratio $w$ is regular as much as $u$ and $v$ are; that is, real analytic. In case of variable coefficients, solutions are not necessarily analytic and then the presence of the free boundary $Z(u)$ deeply affects the analysis. Some partial regularity results, in case the variable coefficients belong to $C^{k,\alpha}$ spaces, are contained in \cite{LinLin22,TerTorVit22}.

Let us assume that the variable coefficients $A$ are symmetric, satisfy \eqref{eq:assump-coeffi} and belong to $C^{k,\omega}(B_1)$ for some $k\in\mathbb N$, $\omega$ a $j$-Dini function with $j\ge 3$. Then the previous result implies Schauder estimates for the ratio $w=v/u$ across the regular set $R(u):=\{x\in Z(u) \, : \, |\nabla u(x)|\neq0\}$. The singular set is denoted by $S(u):=\{x\in Z(u) \, : \, |\nabla u(x)|=0\}$. Let us remark here that the Dini assumptions we make on coefficients ensure local $C^{k+1,\sigma}$ regularity of solutions $u$, with $\sigma$ a $(j-1)$-Dini function, see Corollary \ref{cor:reg:k}, and Dini's implicit function Theorem implies local $C^{k+1,\sigma}$ regularity of the $(n-1)$-dimensional hypersurface $R(u)$. Then, applying Theorem \ref{thm:Dini-bdry-harnack} from both sides of $R(u)$ together with a gluing lemma, one can get

\begin{corollary}[Schauder estimates for the ratio across the regular set]\label{schauderR(u)}
Let $k\in\mathbb N$, $j\in\mathbb N\setminus\{0,1,2\}$, and $\omega$ be a $j$-Dini function. Let us consider two $L$-harmonic functions $u,v\in H^1(B_1)$ with $L$ as in \eqref{eq:div-operator} and $A$ symmetric and satisfying \eqref{eq:assump-coeffi} in $B_1$, $A\in C^{k,\omega}(B_1)$. Let us assume that $S(u)\cap B_1=\emptyset$ and $Z(u)\subseteq Z(v)$. Then, $w=v/u$ belongs to $C^{k+1,\sigma}_\loc(B_1)$ where $\sigma$ is $(j-3)$-Dini, and satisfies the following boundary condition
\begin{equation}\label{boundaryzero}
\mean{A\nabla w,\nu}=0\qquad\mathrm{on \ } R(u)\cap B_1,
\end{equation}
where $\nu$ is the unit normal vector on $R(u)$. Moreover, let $\|A\|_{C^{k,\omega}(B_1)}\leq L_1$ and $u$ be $L$-harmonic in $B_1$ with $S(u)\cap B_1=\emptyset$. Then, for any $L$-harmonic function $v$ in $B_1$ with $Z(u)\subseteq Z(v)$, we have
\begin{equation*}
\left\| \frac{v}{u} \right\|_{C^{k+1,\sigma}(B_{1/2})}\leq C\left\|v\right\|_{L^2(B_1)},
\end{equation*}
with a positive constant $C$ depending on $n,\Lambda,\omega,L_1,u$ and its nodal set $Z(u)$.
\end{corollary}

As we have already remarked, the previous results can be implied by Schauder estimates for solutions to \eqref{e:r}. The weighted Sobolev space $H^1(\Omega\cap B_1,u^2dx)$ is defined as the completion of $C^\infty$ functions with respect to the weighted norm
$$\|w\|_{H^1(\Omega\cap B_1,u^2dx)}^2=\int_{\Omega\cap B_1}u^2(w^2+|\nabla w|^2)dx.$$
Notice that the weight $u^2$ is locally integrable. Then, after a standard flattening of the boundary $\partial\Omega$, everything is reduced to proving regularity up to the characteristic hyperplane $\Sigma:=\{x_n=0\}$ for solutions to
\begin{equation}\label{evenLa3}
-\div\left(x_n^aA\nabla w\right)=\div\left(x_n^a\bff\right)\qquad\mathrm{in \ } B_1^+,
\end{equation}
with $a=2$, $A,\bff\in C^{k,\omega}(B_1^+)$ for some $j$-Dini function $\omega$, and formally satisfying a weighted Neumann boundary condition on $\Sigma$
\begin{equation}\label{NeumannWBC}
\lim_{x_n\to0^+}x_n^a\mean{A\nabla w+\bff,\vec e_n}=0.
\end{equation}
In this paper, we will deal with general powers $a>-1$, and we remark that the $A_2$-Muckenhoupt range $a\in(-1,1)$ has important connections with fractional laplacians of order $s=(1-a)/2\in(0,1)$ in terms of Dirichlet-to-Neumann maps \cite{CafSil07}. The equation \eqref{evenLa3} is degenerate elliptic when $a$ is positive and singular when $a$ is negative. The unit half ball $B_1^+=B_1\cap\{x_n>0\}$ has topological boundary given by $\partial^+B_1^+=\partial B_1\cap\{x_n>0\}$ and $B_1'=B_1\cap\Sigma$. Solutions to \eqref{evenLa3} must be understood as functions belonging to $H^{1,a}(B_1^+):=H^{1}(B_1^+,x_n^adx)$ satisfying the following weak formulation
\begin{equation*}
-\int_{B_1^+}x_n^a\mean{A\nabla w,\nabla\phi}dx=\int_{B_1^+}x_n^a\mean{\bff,\nabla\phi}dx
\end{equation*}
for every test function $\phi\in C^{\infty}_0(B_1)$, actually when $a\geq1$ test functions can be taken in $C^{\infty}_0(B_1^+)$ due to the strong degeneracy of the weight \cite[Proposition 2.2]{SirTerVit21a}. 

\begin{theorem}\label{teok}
Let $a>-1$, $k\in\mathbb N$, $j\in\mathbb N\setminus\{0,1\}$ and $\omega$ be a $j$-Dini function. Let us consider $w\in H^{1}(B_1^+,x_n^adx)$ a weak solution to \eqref{evenLa3} with $A$ symmetric and satisfying \eqref{eq:assump-coeffi} in $B_1^+$, $\bff,A\in C^{k,\omega}(B_1^+)$. Then, $w\in C^{k+1,\sigma}_\loc(B_1^+\cup B_1')$  where $\sigma$ is $(j-2)$-Dini, and satisfies
\begin{equation}\label{Neumann}
\mean{A\nabla w+\bff,\vec e_n}=0\qquad\mathrm{on \ } B'_1.
\end{equation}
Moreover, if $\|A\|_{C^{k,\omega}(B_1^+)}\leq L_1$, then for $0<r<1$, there exists a constant $c>0$, depending on $n$, $\Lambda$, $\omega$, $r$, $k$, $a$ and $L_1$, such that
\begin{equation*}
\|w\|_{C^{k+1,\sigma}(B_r^+)}\leq c\left(\|w\|_{L^2(B_1^+,x_n^adx)}+\|\bff\|_{C^{k,\omega}(B_1^+)}\right).
\end{equation*}
\end{theorem}
Notice that the previous result implies a true Neumann boundary condition \eqref{Neumann} at $\Sigma$ and compare with the weighted one in \eqref{NeumannWBC}.

We remark that the $C^{k+1,\sigma}$ estimate in Theorem~\ref{teok} can be obtained without assuming the symmetry on $A$. Specifically, to prove this theorem, we rely on the $C^2$ estimate for the solution of a homogeneous equation with frozen coefficient (see Proposition~\ref{prop:freeze-coeffi-poly-approx}). This estimate, which does not require the symmetry condition, can be found in \cite{DonPha20}.

We also would like to remark that, in case of classic $C^{k,\alpha}$ Schauder estimates, i.e. for homogeneous power-type moduli, the previous result was proved both in \cite{TerTorVit22} for general powers $a>-1$ and in \cite{Zha23} for $a=2$. The two proofs differ mainly in the way the $C^1$ estimate is obtained. Then, $C^k$ estimates follow by a bootstrap argument. Regarding the gradient estimate, we follow a Campanato iteration argument used also in \cite{Zha23} since it is more effective in the present case of non homogeneous moduli.

\subsection*{Structure of the paper}
The paper is organized as follows: in Section \ref{sec:2} we list some properties of Dini functions and in Proposition \ref{alphaincreasing} we remark that the $\alpha$-nonincreasing condition $(D3)$ in Definition \ref{dinidef} is actually natural. Then, in Corollary \ref{cor:reg} we discuss about the modulus of continuity of $\nabla u$ in the $C^1$ estimate \cite{DonEscKim18}*{Proposition~2.7}. Finally, bootstrapping the latter result for higher derivatives, we imply Corollary \ref{cor:reg:k}; that is, Schauder estimates for uniformly elliptic equations with coefficients, data and boundaries in $C^{k,\omega}$ spaces ($\omega$ $j$-Dini), of independent interest.

Section \ref{sec:3} is devoted to the proof of the main Theorem \ref{teok}; that is, regularity for solutions to the degenerate or singular equation \eqref{evenLa3} for general powers $a>-1$. First, in Lemma \ref{lem:Poincare} we prove a Poincaré type inequality which will be a fundamental tool in our weighted functional framework. Then, a Campanato iteration argument will provide the $C^1$ estimate in Theorem \ref{thm:Dini-ratio-reg}. A special care must be taken during the polynomial approximation in order to obtain the boundary condition \eqref{Neumann}. Eventually, the $C^k$ Schauder estimates are obtained by induction.

In Section \ref{sec:4}, as a consequence of the previous results, we prove the higher order boundary Harnack principle in Theorem \ref{thm:Dini-bdry-harnack} after a standard flattening of the boundary. Finally, in Corollary \ref{schauderR(u)} we imply Schauder estimates for ratios $w=v/u$ of $L$-harmonic functions sharing zero sets $Z(u)\subseteq Z(v)$ across the regular set $R(u)$. In order to glue together the estimates form both sides of $R(u)$, the validity of the Neumann condition \eqref{Neumann} will be crucial.

\subsection*{Notations}\label{subsec:notation}
We use the following notation in this paper.
\begin{itemize}
\item $\R^n$ stands for the $n$-dimensional Euclidean space. We indicate the points in $\R^n$ by $x=(x',x_n)$, where $x'=(x_1,\cdots,x_{n-1})\in \R^{n-1}$, and identify $\R^{n-1}$ with $\R^{n-1}\times\{0\}$.

For $x\in\R^n$ and $r>0$, we let
\begin{alignat*}{2}
  B_r(x)&=\{y\in \R^n:|x-y|<r\},&\quad&\text{ball in $\R^n$,}\\
  B^{\pm}_r(x')&=B_r(x',0)\cap \{\pm y_n>0\},&& \text{half-ball}\\
  B'_r(x')&=B_r(x',0)\cap \{y_n=0\}, &&\text{thin ball.}
\end{alignat*}
When the center is the origin, we simply write $B_r=B_r(0)$, $B_r^\pm=B_r^\pm(0)$ and  $B_r'=B'_r(0)$.

\item The notation $\mean{\cdot,\cdot}$ stands for the scalar product of two vectors; that is, for $x=(x_1,...,x_n)$ and $y=(y_1,...,y_n)$, then $\mean{x,y}=\sum_{i=1}^nx_iy_i$. The orthonormal basis of $\R^n$ is denoted by $\vec e_i$, $i=1,...,n$; that is $\mean{x,\vec e_i}=x_i$.

\item We denote the set of nonnegative integers by $\mathbb{N}=\{0,1,2,3,\cdots\}$.
%Then, fixed a $j\in\mathbb N\setminus\{0\}$, by $\mathbb N_{\ge j}=\mathbb N\setminus\{0,...,j-1\}$.
\item Let $\beta=(\beta_1,...,\beta_n)\in\mathbb N^n$ be a multiindex. Given $|\beta|=\sum_{i=1}^n\beta_i$ and $\partial x^\beta=\prod_{i=1}^n\partial x_i^{\beta_i}$, then the $\beta$ partial derivative of order $|\beta|$ is given by
$$D^\beta u=\frac{\partial^{|\beta|}u}{\partial x^\beta}.$$
When we write $D^ku$ for some $k\in\mathbb N$ then we mean a generic partial derivative $D^\beta u$ with $|\beta|=k$.

\item A modulus of continuity is a function $\omega$ satisfying $(D1)$ in Definition \ref{dinidef}. A $0$-Dini function is a function $\omega$ satisfying $(D1),(D3)$ in Definition \ref{dinidef}. A $j$-Dini function is a function as in Definition \ref{dinidef}.
Given a modulus of continuity $\sigma$ and constants $0<\ka<1$, $0<\al<1$, we use the following notation for a modulus of continuity associated to $\sigma$
\begin{align}
\tilde\sigma(r)=\tilde \sigma_{\al,\kappa}(r)=\sum_{i=0}^\infty\ka^{i\al}\left(\sigma(\ka^{-i}r)[\ka^{-i}r\le1]+\sigma(1)[\ka^{-i}r>1]\right),\label{tildesigma}
\end{align}
where $[\,\cdot\,]$ is Iverson's bracket notation, i.e., $[P]=1$ if $P$ is true, while $[P]=0$ otherwise. For $\cJ_\sigma^j$ as in \eqref{Jomega}, $j=1,2$, we also write for simplicity
\begin{align*}
\cJ_\sigma(r)&=\cJ^1_\sigma(r)=\int_0^r\frac{\sigma(s)}s\,ds,\\
\cH_\sigma(r)&=\cJ^2_\sigma(r)=\int_0^r\frac1s\int_0^s\frac{\sigma(\rho)}\rho\,d\rho ds
\end{align*}
when $\sigma$ is $1$-Dini and $2$-Dini, respectively.

\end{itemize}

\medskip

\section{Some remarks on uniformly elliptic problems with Dini type assumptions}\label{sec:2}
In this section we discuss some properties on Dini functions and regularity results for uniformly elliptic equations with Dini type assumptions on data.

\subsection{Some properties of $j$-Dini functions}

\begin{remark}\label{remarkJalpha}
We list some useful properties of $j$-Dini moduli:
\begin{itemize}
\item[$(1)$] Let $j\in\mathbb N\setminus\{0\}$. If $\sigma$ is $j$-Dini, then $\cJ^j_\sigma$ is a modulus of continuity; that is, a function satisfying $(D1)$.
\item[$(2)$] For any $j\in\mathbb N\setminus\{0\}$, $\cJ^j_\sigma=\cJ^{i}_{\cJ^{l}_\sigma}$ for any $i,l\in\mathbb N$ such that $i+l=j$.
\item[$(3)$] If $\cJ^j_\sigma$ satisfies $(D2)$ for some $j\in\mathbb N\setminus\{0\}$, also $\cJ^{j-1}_\sigma$ does it. This implies $\mathcal D_j\subseteq \mathcal D_{j-1}$.
\item[$(4)$] If $\sigma$ satisfies condition $(D3)$ in Definition \ref{dinidef} for some $\alpha\in(0,1]$, then it satisfies the same condition for any $\beta\in[\alpha,1]$.
\item[$(5)$] Let $j\in\mathbb N\setminus\{0\}$. If $\sigma$ is $j$-Dini satisfying condition $(D3)$ for some $\alpha\in(0,1]$ then $\mathcal J^i_\sigma$ does it with the same $\alpha$ for any $0\leq i\leq j$. In fact
for $0<r<R$
\begin{align*}
\cJ^1_\sigma(r)&=\int_0^r\frac{\sigma(s)}s\,ds=\int_0^R\frac{\sigma\left(\frac{r}Rt\right)}{t}\,dt\ge \int_0^R(r/R)^\al\frac{\sigma(t)}t\,dt=(r/R)^\al\cJ^1_\sigma(R).
\end{align*}
Then one can obtain the same property for any $2\leq i\leq j$ recursively.
\item[$(6)$] Let $j\in\mathbb N\setminus\{0\}$. For a $j$-Dini function $\sigma$ it holds
$$\cJ^{j-1}_\sigma(r)\le \cJ^j_\sigma(r), \qquad 0<r<1.$$
In fact, property $(D3)$ for some $\alpha\in(0,1)$ gives $(D3)$ for $\alpha=1$ by $(4)$; that is, the monotonicity of $\sigma(r)r^{-\al}$ implies that of $\sigma(r)r^{-1}$, which readily gives
$$
\cJ^1_\sigma(r)=\int_0^r\frac{\sigma(\rho)}\rho\,d\rho\ge\int_0^r\frac{\sigma(r)}r\,d\rho=\sigma(r).
$$
This is the first inequality. The other ones follow iteratively applying $(5)$.
\item[$(7)$] Let $j\in\mathbb N\setminus\{0\}$. If $\sigma$ is $j$-Dini, then for any $0\leq i\leq j$, $\cJ^i_\sigma$ is $(j-i)$-Dini. In fact, for a given $0\leq i\leq j$ by combining the properties above it is easy to check that $\cJ^{i}_\sigma$ satisfies $(D1),(D3)$. Then, $\cJ^{(j-i)-1}_{\cJ^i_\sigma}$ satisfies $(D2)$ since
$$\cJ^{(j-i)-1}_{\cJ^i_\sigma}=\cJ^{j-1}_\sigma.$$
\item[$(8)$] Let $\sigma$ be a $1$-Dini function and $c_0\in(0,1]$, $\lambda\in(0,1)$ be constants. Then, for any $k\in \mathbb{N}$,
$$
\sum_{i=k}^\infty\sigma(c_0\la^i)\le \frac{\cJ^1_\sigma(c_0\la^k)}{1-\la}.
$$
Indeed, the condition $(D3)$ for $\al=1$ gives for every $i\in\mathbb{N}$
\begin{align*}
    \sigma(c_0\la^i)&=\frac1{c_0\la^i-c_0\la^{i+1}}\int_{c_0\la^{i+1}}^{c_0\la^i}\sigma(c_0\la^i)\,dt=\frac1{1-\la}\int_{c_0\la^{i+1}}^{c_0\la^i}\frac{\sigma(c_0\la^i)}{c_0\la^i}\,dt\\
    &\le \frac1{1-\la}\int_{c_0\la^{i+1}}^{c_0\la^i}\frac{\sigma(t)}{t}\,dt.
\end{align*}
This readily yields
$$
\sum_{i=k}^\infty\sigma(c_0\la^i)\le\frac1{1-\la}\int_0^{c_0\la^k}\frac{\sigma(t)}t\,dt=\frac{\cJ^1_\sigma(c_0\la^k)}{1-\la}.
$$

\end{itemize}
\end{remark}

First, we would like to prove that condition $(D3)$ in Definition \ref{dinidef} is not restrictive. This remark comes from \cite{ApuNaz16,ApuNaz22,DonEscKim18,DonLeeKim20} where the fact is proved in some subcases. However, we present a unified proof that cover the full range $\alpha\in(0,1]$ and fits with our Definition \ref{dinidef} for any $j\in\mathbb N\setminus\{0\}$. The proof follows some arguments in \cite{ApuNaz22}*{Remark~1.2}.

\begin{proposition}\label{alphaincreasing}
Let $j\in\mathbb N\setminus\{0\}$ and $\sigma$ be a $j$-Dini function as in Definition \ref{dinidef} but without requiring $(D3)$. Then, fixed any $\alpha\in(0,1]$ the function $\sigma_\alpha:[0,1]\to\R$
\begin{equation}\label{sigmaA}
\sigma_\alpha(r)=r^\alpha\sup_{s\in[r,1]}\frac{\sigma(s)}{s^\alpha}
\end{equation}
satisfies
\begin{itemize}
\item[i)] $\sigma(r)\le\sigma_\alpha(r)$;
\item[ii)] $\sigma_\alpha$ satisfies $(D1)$, $(D3)$ with the same $\alpha\in(0,1]$ and $\cJ^{j-1}_{\sigma_\alpha}$ satisfies $(D2)$.
\end{itemize}
Moreover, if $\alpha\in(0,1)$, $\sigma_\alpha$ is $j$-Dini.
\end{proposition}
\begin{proof}
Point $i)$ is trivial. In order to prove point $ii)$, we have to prove that $\sigma_\alpha$ satisfies $(D1), (D3)$ and $\cJ^{j-1}_{\sigma_\alpha}$ satisfies $(D2)$.

We remark that by definition of $\sigma_\alpha$ the property $(D3)$ is trivially satisfied with the chosen $\alpha$. Let us start with proving $(D1)$. First we notice that point $i)$ implies $\sigma_\alpha(r)\ge\sigma(r)>0$ for $r>0$. Then, let us show that $\lim_{r\to0^+}\sigma_\alpha(r)=0$. Then, $\sigma_\alpha$ can be extended continuously up to $0$ with $\sigma_\alpha(0)=0$. Being the case $\sup_{s\in[0,1]}\frac{\sigma(s)}{s^\alpha}=\frac{\sigma(\overline s)}{\overline s^\alpha}$ for a certain $\overline s\in(0,1]$ trivial, we can suppose that, along any sequence $r_n\to0^+$, there exists a sequence $s_n\to0^+$ which realizes the supremum; that is
$$
\sup_{s\in[r_n,1]}\frac{\sigma(s)}{s^\alpha}=\frac{\sigma(s_n)}{s_n^\alpha}.
$$
Then, since $r_n\leq s_n\to0^+$ we have
$$
\sigma_\alpha(r_n)=r_n^\alpha\frac{\sigma(s_n)}{s_n^\alpha}\leq \sigma(s_n)\to0.
$$
For property $(D1)$ it remains to prove that $\sigma_\alpha$ is nondecreasing. Let us remark here that $\sigma_\alpha=\sigma$ in any point of $[0,1]$ for which
$$
\sup_{s\in[r,1]}\frac{\sigma(s)}{s^\alpha}=\frac{\sigma(r)}{r^\alpha}.
$$
Then, the set $\{r\in(0,1)\,:\,\sigma(r)<\sigma_\alpha(r)\}$, which is open in view of the continuity of $\sigma$ and $\sigma_\alpha$, can be expressed as a union of disjoint open intervals $I_i=(\tau_{i1},\tau_{i2})$ with $i\in \mathcal A_1$ at most countable. Moreover, for each $i$ one has
$$\frac{\sigma_\alpha(r)}{r^\alpha}\equiv c_i=\frac{\sigma(\tau_{i2})}{\tau_{i2}} \quad \mathrm{in \ }I_i.$$
Hence, it follows that $\sigma_\alpha$ is nondecreasing.

Let us prove now that $\cJ^{j-1}_{\sigma_\alpha}$ satisfies $(D2)$. Let us start with the case $j=1$; that is, $\cJ^{0}_{\sigma_\alpha}=\sigma_\alpha$ satisfies $(D2)$ . By the previous considerations
\begin{equation*}
\int_0^1\frac{\sigma_\alpha(r)}{r}=\int_{(0,1)\cap\{\sigma_\alpha=\sigma\}}\frac{\sigma(r)}{r}+\sum_{i\in\mathcal A_1}c_i\int_{\tau_{i1}}^{\tau_{i2}}r^{\alpha-1}.
\end{equation*}
The first integral is finite by the Dini condition on $\sigma$; that is, the fact that $\sigma$ satisfies $(D2)$. Moreover, since $c_i=\frac{\sigma(\tau_{i1})}{\tau_{i1}}$ if $\tau_{i1}>0$ and $\sigma(\tau_{i1})=0$ if $\tau_{i1}=0$, the second piece equals
\begin{equation*}
\frac{1}{\alpha}\sum_{i\in\mathcal A_1}(\sigma(\tau_{i2})-\sigma(\tau_{i1}))\leq\frac{\sigma(1)-\sigma(0)}{\alpha}=\frac{\sigma(1)}{\alpha}.
\end{equation*}
In the previous inequality we are using the monotonicity of $\sigma$.

Let us consider now the case $j\geq2$. Let us take $\mathcal J^1_{\sigma_\alpha}(r)=\int_0^r\frac{\sigma_\alpha(s)}{s}$.
%First, by Remark \ref{remarkJalpha} we know that $\cJ^1_{\sigma_\alpha}$ satisfies condition $(D3)$.
Then, arguing as above, for any fixed $r\in(0,1]$ we can consider in $(0,r)$ the sets where $\sigma_\alpha=\sigma$ and $\sigma<\sigma_\alpha$. Hence, the open set where $\sigma<\sigma_\alpha$ in $(0,r)$ is given by the union of disjoint open intervals $I_i(r)=(\tau_{i1}(r),\tau_{i2}(r))$ with $i\in\mathcal A_r$ at most countable. Moreover, for each $i$ one has
$$\frac{\sigma_\alpha(s)}{s^\alpha}\equiv c_i(r)=\frac{\sigma(\tau_{i2}(r))}{\tau_{i2}(r)} \quad \mathrm{in \ }I_i(r).$$
Then

\begin{equation*}
\mathcal J^1_{\sigma_\alpha}(r)=\int_0^r\frac{\sigma_\alpha(s)}{s}=\int_{(0,r)\cap\{\sigma_\alpha=\sigma\}}\frac{\sigma(s)}{s}+\sum_{i\in\mathcal A_r}c_i(r)\int_{\tau_{i1}(r)}^{\tau_{i2(r)}}s^{\alpha-1}.
\end{equation*}
Hence, the first piece is controlled by $\mathcal J^1_\sigma(r)$. The second one, since $c_i(r)=\frac{\sigma(\tau_{i1}(r))}{\tau_{i1}(r)}$ if $\tau_{i1}(r)>0$ and $\sigma(\tau_{i1}(r))=0$ if $\tau_{i1}(r)=0$, by
\begin{equation*}
\frac{1}{\alpha}\sum_{i\in\mathcal A_r}\left(\sigma(\tau_{i2}(r))-\sigma(\tau_{i1}(r))\right)\leq\frac{\sigma(r)-\sigma(0)}{\alpha}=\frac{\sigma(r)}{\alpha}=\frac{\cJ^{0}_\sigma(r)}{\alpha}.
\end{equation*}
Hence, by $(2)$ in Remark \ref{remarkJalpha},
$$\cJ^{j-1}_{\sigma_\alpha}=\cJ^{j-2}_{\cJ^1_{\sigma_\alpha}}\le \cJ^{j-2}_{\cJ^1_{\sigma}}+\cJ^{j-2}_{\frac{\cJ^0_{\sigma}}{\alpha}}=\cJ^{j-1}_{\sigma}+\frac{1}{\alpha}\cJ^{j-2}_{\sigma}.$$
Then, by definition $\cJ^{j-1}_\sigma$ satisfies $(D2)$ and by $(3)$ also $\cJ^{j-2}_\sigma$ does it.
%Hence, $\mathcal J_{\sigma_\alpha}$ is $(j-1)$-Dini since it is controlled pointwisely by the $(j-1)$-Dini function $\mathcal J_{\sigma}$ and the $j$-Dini function $\sigma$.
\end{proof}

We remark that some portion of the proof of the above proposition can be obtained alternatively by using the result in Lemma~\ref{lem:j-Dini}. For example, for a $j$-Dini function $\sigma$, one can show (see \cite{DonEscKim18}) $\sigma_\alpha\le C\tilde\sigma_{\al,1/2}$, where $\tilde\sigma_{\al,1/2}$ is as in \eqref{tildesigma}. Since $\tilde\sigma_{\al,1/2}$ is $j$-Dini by Lemma~\ref{lem:j-Dini}, this implies that $\cJ_{\sigma_\al}^{j-1}$ satisfies the integrability condition $(D2)$.

From now on, in view of the previous remarks, we can consider $j$-Dini functions $\sigma$ satisfying $(D3)$ for some chosen $\alpha\in(0,1)$. Hence, $\sigma$ coincides with $\sigma_\alpha$ in \eqref{sigmaA}.

%\begin{assumption}
%Any Dini type modulus $\omega$ we will consider will be a $j$-Dini function for some $j\in\mathbb{N}\setminus\{0\}$ such that:
%\begin{itemize}
%\item[i)] there exists $\alpha\in(0,1)$ such that $\omega(r)r^{-\alpha}$ is not increasing.
%\item[ii)] {\color{blue} maybe other properties?}
%\end{itemize}
%\end{assumption}

\begin{example}
For any $j\in\mathbb N\setminus\{0\}$, one has $\mathcal D_{j+1}\subsetneq \mathcal D_{j}$. An example of $j$-Dini function which is not $(j+1)$-Dini in $[0,1]$ is given by
\begin{equation*}
\omega_j(r)=\frac{1}{(-\log r+1)^{j+1}}.
\end{equation*}
It trivially satisfies $(D1)$. Actually, in order to enjoy condition $(D3)$ for a certain $\alpha\in(0,1)$ one can define $\omega_{j,\alpha}$ as in \eqref{sigmaA}. Let us check the validity of condition $(D2)$ for $\cJ^{j-1}_{\omega_j}$, which follows by proving by induction on $j\in\mathbb N\setminus\{0\}$ that
\begin{equation}\label{formulaLog}
\cJ^{j-1}_{\omega_j}=\frac{\omega_1}{j!}.
\end{equation}
The formula above holds trivially when $j=1$. Then, let us suppose the formula true for a generic $j\in\mathbb N\setminus\{0\}$ and prove it for $j+1$. By direct integration, one notice that
$$
\cJ^1_{\omega_j}=\frac{\omega_{j-1}}{j}.
$$
Then
\begin{equation*}
\cJ^{(j+1)-1}_{\omega_{j+1}}=\cJ^{j-1}_{\cJ^1_{\omega_{j+1}}}=\cJ^{j-1}_{\frac{\omega_j}{j+1}}=\frac{1}{j+1}\cJ^{j-1}_{\omega_j}=\frac{\omega_1}{(j+1)!}.
\end{equation*}
Hence $\omega_j\in\mathcal D_j$. Let us remark that \eqref{formulaLog} also says that $\cJ^j_{\omega_j}$ does not satisfy $(D2)$ anymore; that is, $\omega_j\not\in\mathcal D_{j+1}$.
\end{example}

The following auxiliary lemma will be needed in the next part of this section.

\begin{lemma}
    \label{lem:j-Dini}
For $j\in\mathbb N\setminus\{0\}$, if $\sigma$ is a $j$-Dini function satisfying $(D3)$ for a chosen $\al\in(0,1)$, then also $\tilde\sigma=\tilde\sigma_{\al,\kappa}$ defined in \eqref{tildesigma} is $j$-Dini for any $0<\ka<1$.
\end{lemma}

\begin{proof}
We divide our proof into three steps, one for each condition $(D1)$, $(D2)$, $(D3)$ in Definition~\ref{dinidef}.

\medskip\noindent\emph{Step 1.} We prove $(D1)$. Clearly, $\tilde\sigma$ is nondecreasing, $\tilde\sigma(0)=\sum_{i=0}^\infty \ka^{i\al}\sigma(0)=0$ and by definition $\tilde \sigma(r)>0$ for $r>0$. Moreover, $\tilde\sigma$ is continuous as it is a uniform limit of a sequence of continuous functions.

\medskip\noindent\emph{Step 2.}
In this step we prove $(D2)$. We first note that 
$$
\tilde\sigma(t)\le\sum_{i=0}^\infty\ka^{i\al}\sigma(\ka^{-i}t)[\ka^{-i}t\le1]+Nt^\gamma
$$
for some constants $N>0$ and $\gamma>0$, depending on $\al$, $\ka$, $\sigma(1)$. As the Hölder continuous function $t\longmapsto Nt^\gamma$ clearly satisfies $(D2)$ together with any $\cJ_{Nt^\gamma}^i=Nt^\gamma /\gamma^i$, we only need to prove that $\cJ_{\bar\sigma}^j(1)<\infty$, where $\cJ_{\bar\sigma}^j$ is as in \eqref{Jomega} with $\bar\sigma(t)=\sum_{i=0}^\infty\ka^{i\al}\sigma(\ka^{-i}t)[\ka^{-i}t\le1]$. To this aim, we observe that
\begin{align}
    \label{eq:k-dini}
    \begin{split}
&\cJ_{\bar\sigma}^j(1)\\    
&=\int_0^1\frac1{r_1}\int_0^{r_1}\frac1{r_2}\int_0^{r_2}\cdots\frac1{r_{j-1}}\int_0^{r_{j-1}}\sum_{i=0}^\infty\frac{\ka^{i\al}\sigma(\ka^{-i}r_j)}{r_j}[\ka^{-i}r_j\le1]dr_j\cdots dr_3dr_2dr_1\\
&=\sum_{i=0}^\infty\ka^{i\al}\left(\int_0^1\frac1{r_1}\int_0^{r_1}\frac1{r_2}\int_0^{r_2}\cdots\frac1{r_{j-1}}\int_0^{r_{j-1}}\frac{\sigma(\ka^{-i}r_j)}{r_j}[\ka^{-i}r_j\le1]dr_j\cdots dr_3dr_2dr_1\right)\\
&=:\sum_{i=0}^\infty\ka^{i\al}\mathbb{A}^j_i.
\end{split}\end{align}
Now we compute $\mathbb{A}^j_i$ for $i\ge0$, $j\ge1$. For this purpose, we consider a function $\cA^j_i(t)$, $0<t\le1$, defined by 
\begin{align*}
    \begin{cases}
    \cA^1_i(t):=\int_0^t\frac{\sigma(\ka^{-i}r_1)}{r_1}[\ka^{-i}r_1\le1]\,dr_1,&j=1,\\
    \cA^j_i(t):=\int_0^t\frac1{r_1}\int_0^{r_1}\frac1{r_2}\int_0^{r_2}\cdots\frac1{r_{j-1}}\int_0^{r_{j-1}}\frac{\sigma(\ka^{-i}r_j)}{r_j}[\ka^{-i}r_j\le1]dr_j\cdots dr_3dr_2dr_1,& j\ge2.
    \end{cases}
\end{align*}
Notice that $\mathbb{A}^j_i=\cA^j_i(1)$. By applying the change of variable formula repeatedly, one can get 
\begin{align}
    \label{eq:k-dini-ftn}
    \begin{cases}
\cA^1_i(t)=\int_0^{\ka^{-i}t}\frac{\sigma(s_1)}{s_1}[s_1\le1]\,ds_1,&j=1,\\
\cA^j_i(t)=\int_0^{\ka^{-i}t}\frac1{s_1}\int_0^{s_1}\frac1{s_2}\int_0^{s_2}\cdots\frac1{s_{j-1}}\int_0^{s_{j-1}}\frac{\sigma(s_j)}{s_j}[s_j\le1]\,ds_{j-1}\cdots ds_3ds_2ds_1,&j\ge2.
    \end{cases}
\end{align}

We now claim that 
\begin{align}
    \label{eq:k-dini-induc}
    \cA^j_i(t)=\begin{cases}
        \cJ^j_\sigma(\ka^{-i}t),&0<t\le \ka^i,\\
        \sum_{m=1}^j\frac{\cJ^m_\sigma(1)}{(j-m)!}\left(\log(t/\ka^i)\right)^{j-m},&\ka^i<t\le1.
    \end{cases}
\end{align}
Indeed, if $0<t\le\ka^i$, then $\ka^{-i}t\le 1$, thus the equality $\cA^j_i(t)=\cJ^j_\sigma(\ka^{-i}t)$ simply follows from \eqref{eq:k-dini-ftn} and the definition of $\cJ^j_{\sigma}$ (as the Iverson bracket $[s_j\le1]$ is $1$ in \eqref{eq:k-dini-ftn} if $s_j=\ka^{-i}t\le1$). Thus we focus on \eqref{eq:k-dini-induc} when $t\in(\ka^i,1]$. We prove it by induction on $j\in\mathbb{N}\setminus\{0\}$. When $j=1$, using $\ka^{-i}t>1$, we get 
$$
\cA^1_i(t)=\int_0^{\ka^{-i}t}\frac{\sigma(s_1)}{s_1}[s_1\le 1]\,ds=\int_0^1\frac{\sigma(s_1)}{s_1}\,ds_1=\cJ^1_\sigma(1).
$$
We next assume that \eqref{eq:k-dini-induc} is true for $j$, and prove it for $j+1$. We observe that 
\begin{align*}
    \cA^{j+1}_i(t)&=\int_0^t\frac{\cA^j_i(r)}r\,dr=\int_0^{\ka^i}\frac{\cA^j_i(r)}r\,dr+\int_{\ka^i}^t\frac{\cA^j_i(r)}r\,dr\\
    &=\int_0^{\ka^i}\frac{\cJ^j_\sigma(\ka^{-i}r)}r\,dr+\int_{\ka^i}^t\frac1r\sum_{m=1}^j\frac{\cJ^m_\sigma(1)}{(j-m)!}\left(\log(r/\ka^i)\right)^{j-m}\,dr.
\end{align*}
We compute 
\begin{align*}
    \int_0^{\ka^i}\frac{\cJ^j_\sigma(\ka^{-i}r)}r\,dr=\int_0^1\frac{\cJ^j_\sigma(s)}s\,ds=\cJ^{j+1}_\sigma(1),
\end{align*}
and 
\begin{align*}
    \int_{\ka^i}^t\frac1r\sum_{m=1}^j\frac{\cJ^m_\sigma(1)}{(j-m)!}\left(\log(r/\ka^i)\right)^{j-m}\,dr&=\sum_{m=1}^j\cJ^m_\sigma(1)\int_{\ka^i}^t\frac{\left(\log(r/\ka^i)\right)^{j-m}}{(j-m)!r}\,dr\\
    &=\sum_{m=1}^j\cJ^{m}_\sigma(1)\left[\frac{\left(\log(r/\ka^i)\right)^{j+1-m}}{(j+1-m)!}\right]_{\ka^i}^t\\
    &=\sum_{m=1}^j\frac{\cJ^m_\sigma(1)}{(j+1-m)!}\left(\log(r/\ka^i)\right)^{j+1-m}.
\end{align*}
Combining the above equalities yields,
\begin{align*}
    \cA^{j+1}_i(t)&=\cJ^{j+1}_\sigma(1)+\sum_{m=1}^j\frac{\cJ^m_\sigma(1)}{(j+1-m)!}\left(\log(r/\ka^i)\right)^{j+1-m}\\
    &=\sum_{m=1}^{j+1}\frac{\cJ^m_\sigma(1)}{(j+1-m)!}\left(\log(r/\ka^i)\right)^{j+1-m}.
\end{align*}
This means that \eqref{eq:k-dini-induc} is true for $j+1$, and thus the claim \eqref{eq:k-dini-induc} is proved. 

Now we are ready to complete the proof of $(D2)$. Based on \eqref{eq:k-dini}, $(D2)$ follows once we show that $\sum_{i=0}^\infty\ka^{i\al}\mathbb{A}^j_i<\infty$. To this aim, we use \eqref{eq:k-dini-induc} to have 
\begin{align*}
    \mathbb{A}^j_i&=\cA^j_i(1)=\sum_{m=1}^j\frac{\cJ^m_\sigma(1)}{(j-m)!}\left(\log(1/\ka^i)\right)^{j-m}=\sum_{m=1}^j\frac{\cJ^m_\sigma(1)}{(j-m)!}(-\log \ka)^{j-m}i^{j-m}.
\end{align*}
This, along with the fact that $\sum_{i=0}^\infty i^{j-m}\ka^{i\al}\le C(j,m,\ka,\al)<\infty$ for every $1\le m\le j$, concludes
\begin{align*}
    \sum_{i=0}^\infty\ka^{i\al}\mathbb{A}^j_i=\sum_{m=1}^j\frac{\cJ^m_\sigma(1)}{(j-m)!}(-\log \ka)^{j-m}\left(\sum_{i=0}^\infty i^{j-m}\ka^{i\al}\right)<\infty.
\end{align*}

\medskip\noindent\emph{Step 3.} Finally, we show that $\tilde\sigma$ enjoys $(D3)$. To prove it, let $0<r<R$ be given.

We first consider the simple case that $\ka^{N+1}\le r<R\le \ka^N$ for some nonnegative integer $N$. If $N=0$, then $\ka\le r<R$, thus $\tilde\sigma(r)=\sum_{i=1}^\infty\ka^{i\al}\sigma(1)=\tilde\sigma(R)$, hence $\tilde\sigma(r)\ge(r/R)^\al\tilde\sigma(R)$ holds. Thus, we may assume $N\ge1$. Then we can write 
\begin{align*}
    &\tilde\sigma(r)=\sum_{i=1}^N\ka^{i\al}\sigma(\ka^{-i}r)+\sum_{i=N+1}^\infty\ka^{i\al}\sigma(1),\\
    &\tilde\sigma(R)=\sum_{i=1}^N\ka^{i\al}\sigma(\ka^{-i}R)+\sum_{i=N+1}^\infty\ka^{i\al}\sigma(1).
\end{align*}
This readily gives $\tilde\sigma(r)\ge(r/R)^\al\tilde\sigma(R)$, since $\sigma(\ka^{-i}r)\ge(r/R)^\al\sigma(\ka^{-i}R)$ and $1\ge(r/R)^\al$.

Next, for the general case $0<r<R<1$, not falling into the first simple case, we can find two nonnegative integers $N$ and $M$ such that $\ka^{M+1}<r\le \ka^M<\ka^N\le R<\ka^{N-1}$. Then the first case gives \begin{align*}
    &\tilde\sigma(r)\ge \left(\frac{r}{\ka^M}\right)^\al\tilde\sigma(\ka^M),\quad\tilde\sigma(\ka^M)\ge\left(\frac{\ka^M}{\ka^{M-1}}\right)^\al\tilde\sigma(\ka^{M-1}),\,\cdots\,,\\
    &\tilde\sigma(\ka^{N+1})\ge\left(\frac{\ka^{N+1}}{\ka^N}\right)^\al\tilde\sigma(\ka^N),\quad\tilde\sigma(\ka^N)\ge\left(\frac{\ka^N}{R}\right)^\al\tilde\sigma(R),
\end{align*}
which implies $\tilde\sigma(r)\ge(r/R)^\al\tilde\sigma(R)$.
\end{proof}

\subsection{Schauder estimates for uniformly elliptic problems in Dini type domains}
Aim of this section is to prove general Schauder estimates up to the boundary for uniformly elliptic equations with coefficients, data and boundaries in $C^{k,\omega}$ spaces with $\omega$ a $j$-Dini function. The result is based on a bootstrap argument for higher derivatives of a given solution. Hence, as a starting point, we need to discuss about the following $C^1$ regularity result provided in \cite{DonEscKim18}*{Proposition~2.7}.

We would like to stress the fact that the results contained in the present section hold true even without requiring the symmetry of variable coefficients $A=A^T$.

\begin{proposition}\label{prop:reg}
    Suppose $A$ and $\bg$ are of Dini mean oscillation. If $u\in H^{1}(B_1^+)$ is a weak solution of 
    \begin{equation}\label{EQHD}
    \begin{cases}
    -\div(A\D u)=\div \bg  &\text{in }B^+_{1}\\
    u=0 &\text{on }B'_1,
    \end{cases}
    \end{equation}
   then $u\in C^1_\loc(B_{1}^+\cup B'_1)$.
\end{proposition}

After a careful analysis of the modulus of continuity $\sigma_u$ of $\nabla u$ (to avoid bulky notation we simply write $u$ instead of $\D u$ in the subscript), one can imply the following

\begin{corollary}\label{cor:reg}
Let $j\in\mathbb N\setminus\{0\}$. Suppose $A$ and $\bg$ are of $j$-Dini mean oscillation. If $u\in H^{1}(B_1^+)$ is a weak solution of \eqref{EQHD}, then $u\in C^{1,\sigma_u}_\loc(B_{1}^+\cup B'_1)$ where the modulus $\sigma_u$ is a $(j-1)$-Dini function.
\end{corollary}

Let us remark that the case $j=1$ corresponds to Proposition \ref{prop:reg} and that the case $j=2$ is studied also in \cite[Theorem 2.5]{LaMLeoSch20} assuming uniform double Dini continuity of data.
To prove Corollary~\ref{cor:reg}, we need to combine Lemma~\ref{lem:j-Dini} with the result in \cite{DonEscKim18}.

\begin{proof}[Proof of Corollary~\ref{cor:reg}]
Based on the equation (2.36) in \cite{DonEscKim18}, for any constant $\al\in(0,1)$ the modulus of continuity $\sigma_u$ of $\D u$ is bounded by
\begin{align}\label{eq:mod-conti-sol-grad}
\sigma_{u}(r)\le C(r^\alpha+\omega_A^*(r)+\omega_\bg^*(r)),\quad 0<r<1, 
\end{align}
for some constant $C>0$, independent of $r$ but depending on $\alpha,\kappa$ and $H^1$-norm of $u$. Here, the moduli of continuity $\omega_A^*$, $\omega_\bg^*$ are defined (in \cite{DonEscKim18}) as follows: we first let 
\begin{equation}\label{sharp}
\hat\omega_A(r):=\tilde\omega_A(r)+\tilde\omega_A(4r)+\omega^\sharp_A(4r),\quad\text{where }  \omega^\sharp_A(r):=
\sup_{s\in[r,1]}\left(\left(\frac rs\right)^\alpha\tilde\omega_A(s)\right).
\end{equation}
$\tilde\omega_A$ is as in \eqref{tildesigma} with $\al$ as above and proper $\ka\in(0,1)$ depending on $\al$.
Then $\omega_A^*(r)$ is
$$
\omega^*_A(r):=\hat\omega_A(r)+\tilde\omega_A(4r)+\int_0^r\frac{\tilde\omega_A(s)}s\,ds+\int_0^r\frac{\tilde\omega_A(4s)}s\,ds.
$$
$\omega_\bg^*$ is defined in a similar way.

In our setting, the moduli $\omega_A^*$, $\omega_\bg^*$ can be simplified. In fact, the $\al$-nonincreasing condition holds for $\tilde\omega_A$ by Lemma~\ref{lem:j-Dini}. It then follows that $\omega_A^\sharp=\tilde\omega_A$ by definition of $\omega_A^\sharp$, and that $\frac{\tilde\omega_A(4r)}{(4r)^\al}\le \frac{\tilde\omega_A(r)}{r^\al}$, which is equivalent to $\tilde\omega_A(4r)\le 4^\al\tilde\omega_A(r)$. Thus we have
$$
\hat\omega_A(r)\le C\tilde\omega_A(r),
$$
which in turn gives 
$$
\omega_A^*(r)\le C\tilde\omega_A(r)+C\int_0^r\frac{\tilde\omega_A(s)}{s}\,ds\le C\int_0^r\frac{\tilde\omega_A(s)}{s}\,ds,
$$
where used in the last inequality the point $(6)$ in Remark~\ref{remarkJalpha}. The above observation implies that we can redefine $\hat\omega_A(r)$, $\omega_A^*$ in \eqref{eq:mod-conti-sol-grad} by 
\begin{align}
    \label{eq:w-*}
    \hat\omega_A(r):=\tilde\omega_A(r),\quad \omega_A^*(r):=\int_0^r\frac{\tilde\omega_A(s)}s\,ds=\cJ_{\tilde\omega_A}(r).
\end{align}
As $\tilde\omega_A$ is $j$-Dini (by Lemma~\ref{lem:j-Dini}), $\omega^*_A(r)$ is $(j-1)$-Dini. We can repeat the above process with $\omega_\bg^*$, and conclude that $\sigma_u$ is a $(j-1)$-Dini function.
\end{proof}

Now we are ready to state and prove boundary Schauder estimates for higher derivatives. Let $0\in\partial\Omega$ and consider a solution $u\in H^1(\Omega\cap B_1)$ to
\begin{equation*}
\begin{cases}
-\div(A\nabla u)=\div \bg &\mathrm{in \ }\Omega\cap B_1\\
u=0 &\mathrm{on \ }\partial\Omega\cap B_1.
\end{cases}
\end{equation*}
Let $k\in\mathbb N$, $j\in\mathbb N\setminus\{0\}$. If $\partial\Omega\in C^{k+1,\omega}$, $A,\bg\in C^{k,\omega}(\Omega\cap B_1)$ with $\omega$ a $j$-Dini function, then we will prove that $u\in C^{k+1,\sigma_u}_\loc(\overline\Omega\cap B_1)$ with $\sigma_u$ a $(j-1)$-Dini function. This is done after a standard flattening on $\partial\Omega$, which reduces everything to prove Schauder estimates for solutions to \eqref{EQHD}. Let us remark that this standard procedure with the straightening diffeomorphism $\Phi$ in \eqref{diffeo} will be introduced and discussed in details later in Section \ref{sec:4}.

\begin{corollary}\label{cor:reg:k}
Let $k\in\mathbb N$ and $j\in\mathbb N\setminus\{0\}$. Suppose $A,\bg\in C^k(\overline{B_1^+})$ with $D^kA$, $D^k\bg$ being of $j$-Dini mean oscillation. If $u\in H^{1}(B_1^+)$ is a weak solution of \eqref{EQHD}, then $u\in C^{k+1,\sigma}_\loc(B_{1}^+\cup B'_1)$ where the modulus $\sigma$ is a $(j-1)$-Dini function.
\end{corollary}
\proof
The proof is based on the induction argument in \cite[Proposition 3.3]{Vit22}. When $k=0$ the result follows by Corollary \ref{cor:reg}. Let us now suppose the result true for a certain $k\in\mathbb N$ and prove it for $k+1$; that is, we assume $A,g\in C^{k+1}(\overline{B_1^+})$ and $D^{k+1}A$, $D^{k+1}\bg$ are of $j$-Dini mean oscillation, and we want to prove $u\in C^{k+2,\sigma}$ for some $(j-1)$-Dini function $\sigma$. We remark that the thesis is equivalent to $u_i:=\partial_i u\in C^{k+1,\sigma}$ for any $i=1,...,n$. It is easy to see that any tangential derivative $u_i$ with $i=1,...,n-1$ is a solution to
$$
 \begin{cases}
 -\div(A\D u_i)=\div (\partial_i\bg+\partial_iA\nabla u)   & \text{in }B^+_{1}\\
 u_i=0 & \text{on }B'_{1}.
 \end{cases}  
 $$
 Let us remark here that the assumption $A,\bg\in C^{k+1}$ implies $A,\bg\in C^{k,1}$ which gives, by standard Schauder estimates that $u\in C^{k+1,\alpha}$ for any $\alpha\in(0,1)$ and hence $\D u\in C^k$ with $D^k(\D u)$ being of $j$-Dini mean oscillation. Then, the inductive hypothesis implies
 \begin{equation}\label{u_i}
 u_i\in C^{k+1,\sigma} \qquad\mathrm{for \  any \ } i=1,...,n-1.
 \end{equation}
 In order to prove that $u_n\in C^{k+1,\sigma}$ it is enough to prove that $u_{nn}:=\partial^2_{nn}u\in C^{k,\sigma}$ being $u_{ni}:=\partial^2_{ni}u\in C^{k,\sigma}$ for any $i=1,...,n-1$ already implied by \eqref{u_i}. Then, for this last partial derivative, one can rewrite equation \eqref{EQHD} obtaining
$$
 u_{nn}=-\frac{1}{a_{nn}}\left(\div \bg+\sum_{i=1}^{n-1}\partial_i(\mean{A\nabla u,\vec e_i})+\sum_{i=1}^{n-1}\partial_n(a_{ni}u_i)+\partial_na_{nn}u_n\right)\in C^{k,\sigma}.
$$
Indeed, from $a_{nn}=\mean{A\vec e_n,\vec e_n}\geq 1/\Lambda$ and $A\in C^{k+1}(\overline{B_1^+})$, we find $\frac1{a_{nn}}\in C^{k+1}(\overline{B_1^+})$, and thus $D^k\left(\frac1{a_{nn}}\right)$ is a $(j-1)$-uniformly Dini continuous function. Concerning the first term $\div \bg$ in the big parenthesis above, we have by the assumption hypothesis that $D^k(\div g)$ is $j$-Dini mean oscillation, and hence $(j-1)$-uniformly Dini continuous (see the proof of Lemma~A.1 in \cite{HwaKim20}). The other terms in the parenthesis can be shown in a similar way to be $C^k$-functions with their $k$-th derivatives being $(j-1)$-uniformly Dini continuous. 
\endproof

%{\color{blue}
%
%\begin{remark}
%The previous result holds true weakening the requirements on $A,\bg$; that is, $A,\bg$ must possess partial derivatives $D^\beta$ of order $|\beta|=k$ which are of $j$-Dini mean oscillation. Defining $C^{k,\omega}_{MO}(\Omega)$ the space of $C^k(\overline\Omega)$ functions such that any partial derivative $D^\beta$ of order $|\beta|=k$ is of $j$-Dini mean oscillation with respect to a given $j$-Dini function $\omega$. Observe that $C^{k,\omega}(\Omega)\subseteq C^{k,\omega}_{MO}(\Omega)$.
%\end{remark}
%
%
%\begin{corollary}\label{cor:reg:k:2}
%Let $k\in\mathbb N$ and $j\in\mathbb N\setminus\{0\}$. Suppose $A,\bg\in C^{k,\omega}_{MO}(B_1^+)$ with modulus $\omega$ a $j$-Dini function. If $u\in H^{1}(B_1^+)$ is a weak solution of \eqref{EQHD}, then $u\in C^{k+1,\sigma_u}_\loc(B_{1}^+\cup B'_1)$ where the modulus $\sigma_u$ is a $(j-1)$-Dini function.
%\end{corollary}
%}

%%%%%%%%%%%%%%%%%%%%%%%%%%%%%%%%%%%%%%%%%%%%%%%%%%%%%%%%%%%%%%

\section{Schauder estimates for degenerate equations with Dini type assumptions}\label{sec:3}
Aim of this section is the proof of the main Theorem \ref{teok}; that is, Schauder estimates for solutions to \eqref{evenLa3} under Dini type assumptions on coefficients and data. We would like to stress the fact that the regularity estimate we get still holds true when coefficients are not symmetric.

\subsection{The $C^1$ estimate}

For $a>-1$, let $w\in H^{1,a}(B_1^+)=H^1(B_1^+,x_n^adx)$ be a solution of \eqref{evenLa3}; that is,
\begin{align*}
-\div(x_n^aA\D w)=\div(x_n^a\bff),
\end{align*}
and formally satisfying a weighted Neumann boundary condition at $\Sigma=\{x_n=0\}$ \eqref{NeumannWBC}; that is,
\begin{equation*}
\lim_{x_n\to0^+}x_n^a\mean{A\nabla w+\bff, \vec e_n}=0.
\end{equation*}
The weighted Sobolev space $H^{1,a}(B_1^+)$ is defined as the completion of $C^\infty(\overline{B_1^+})$ functions with respect to the norm
\begin{equation*}
\|w\|^2_{H^{1,a}(B_1^+)}=\int_{B_1^+}x_n^a\left(w^2+|\nabla w|^2\right).
\end{equation*}
When the weight is \emph{superdegenerate} $a\geq1$, the space $H^{1,a}(B_1^+)$ is actually the completion of $C^\infty_0(\overline{B_1^+}\setminus\Sigma)$ functions, see \cite[Proposition 2.2]{SirTerVit21a}. Then, the weak formulation for \eqref{evenLa3} is given by
\begin{equation*}
-\int_{B_1^+}x_n^a\mean{A\nabla w,\nabla\phi}=\int_{B_1^+}x_n^a\mean{\bff,\nabla\phi}
\end{equation*}
for every test function $\phi\in C^{\infty}_0(B_1)$ (when $a\geq1$ test functions can be taken in $C^{\infty}_0(B_1^+)$).

\begin{remark}
In this section, any time we write that a weighted equation is satisfied on a half ball of type $B_r^+$, we mean that also a Neumann boundary condition is prescribed on the flat boundary $B'_r$.
\end{remark}

In the present section, we are going to prescribe traces on $\partial^+B_1^+$. Hence, we define $H^{1,a}_0(B_1^+\cup B'_1)$ as the completion of $C^\infty_0(B_1^+\cup B'_1)$ functions in the \emph{$A_2$-Muckenhoupt} range $a\in(-1,1)$, and of $C^\infty_0(B_1^+)$ functions in the \emph{superdegenerate} range $a\ge 1$. Let us remark here that in the latter range, due to the zero capacity of $\Sigma$, prescribing traces on $\Sigma$ does not have sense in general. Hence, $h=w$ on $\partial^+ B^+_1$ if and only if $w-h\in H^{1,a}_0(B_1^+\cup B'_1)$. The next result is a weighted Poincaré inequality for functions in $H^{1,a}_0(B_1^+\cup B'_1)$. We would like to remark that the result was known for general \emph{$A_2$-Muckenhoupt} weights (in our setting when $-1<a<1$), see \cite[Theorem 1.2]{FabKenSer82}. The proof we present works in the full range $a>-1$.

\begin{lemma}[Weighted Poincaré Inequality]\label{lem:Poincare}
Let $a>-1$. Then, if $r>0$ and $u\in H^{1,a}_0(B_r^+\cup B'_r)$
\begin{align*}
\frac{(a+1)^2}4\int_{B_r^+}x_n^au^2\le r^2\int_{B_r^+}x_n^a|\D u|^2.
\end{align*}   
\end{lemma}

\begin{proof}
Since $C^\infty_0(B_r^+\cup B_r')$ is dense in $H^{1,a}_0(B_r^+\cup B_r')$, we may assume without loss of generality $u\in C^\infty_0(B_r^+\cup B'_r)$. We then have by integration by parts 
\begin{align*}
    a\int_{B_r^+}x_n^au^2&=\int_{B_r^+}\partial_n(x_n^a)x_nu^2=\int_{\partial(B_r^+)}x_n^{1+a}u^2\mean{\vec e_n,\nu} -\int_{B_r^+}x_n^a\partial_n(x_nu^2)\\
    &=-\int_{B_r^+}x_n^au^2-2\int_{B_r^+}x_n^{a+1}u\partial_nu,
\end{align*}
where we have used in the last step that $x_n^{1+a}u^2=0$ on $\partial(B_r^+)$.
Then, it follows
\begin{align*}
    \frac{a+1}2\int_{B_r^+}x_n^au^2&=-\int_{B_r^+}x_n^{a+1}u\partial_nu\le \left(\int_{B_r^+}x_n^au^2\right)^{1/2}\left(\int_{B_r^+}x_n^{a+2}|\partial_nu|^2\right)^{1/2}.
\end{align*}
Therefore, we conlude 
\begin{equation*}
\frac{(a+1)^2}4\int_{B_r^+}x_n^au^2\le \int_{B_r^+}x_n^{a+2}|\partial_n u|^2\le r^2\int_{B_r^+}x_n^{a}|\D u|^2 .\qedhere  
\end{equation*}  
\end{proof}

Suppose $A,\bff\in C^{0,\sigma}(B_1^+)$ with $\sigma$ a $j$-Dini function with $j\in \mathbb N\setminus\{0,1\}$.
If we set for $0<\delta<1$ 
$$
A^\delta(x):=A(\delta x),\quad \bff^\delta(x):=\bff(\delta x),\quad w^\delta(x):=\frac{w(\delta x)}\delta,
$$
then $w_\delta$ solves \begin{align*}
-\div(x_n^aA^\delta\D w^\delta)=\div(x_n^a\bff^\delta)\quad\text{in }B_1^+.
\end{align*}
Thus, by considering those scalings, we may assume that $$
|A(x)-A(y)|\le L_1\sigma(\delta_0|x-y|),\quad L_1:=\|A\|_{C^{0,\sigma}(B_1^+)},
$$
for some small constant $\delta_0$ to be chosen later. Finally, we assume without loss of generality $\sigma(1)=1$. Throughout this section we say that a constant $C$ is \emph{universal} if it depends only on $n$, $\Lambda$, $\al$, $a$, $L_1$ and $\sigma$.

In what follows, we are going to prove $C^1$ regularity for solutions to \eqref{evenLa3}. In order to do so, we perform a standard freezing of the variable coefficients at a given point $x_0\in B'_1$, without loss of generality $x_0=0$. Then, if the coefficients are symmetric, it is not restrictive to assume that $A(0)=\mathbb I_n$, by a standard linear change of variables. Nevertheless, we do not make the latter assumption, since the regularity estimates we obtain can be proved also when coefficients are not symmetric. Moreover, we would like to keep track of the dependence on $A(0)$. In fact, the validity of the Neumann boundary condition \eqref{Neumann} at $\Sigma$ will come from repeating the freezing with consequent polynomial approximation at any given boundary point $x_0\in B'_1$.

\begin{proposition}
    \label{prop:freeze-coeffi-poly-approx}
    If $h\in H^{1,a}(B_r^+)$ solves $$
    \div(x_n^aA(0)\D h)=0\quad\text{in }B_r^+,
    $$
    then there exists an affine function $l$ such that \begin{align*}
        \int_{B_\rho^+}x_n^a|h-l|^2&\le C\left(\rho/r\right)^{n+4+a}\int_{B_r^+}x_n^ah^2,\quad \rho<r/2,\\
        |l(0)|^2+r^2|\D l|^2&\le \frac{C}{r^{n+a}}\int_{B_r^+}x_n^ah^2,
         \end{align*}
         and
          \begin{align}\label{eqlin}
        \begin{cases}
         \div(x_n^aA(0)\D l)=0 &\text{in }B_r^+,\\
        \mean{A(0)\D l,\vec e_n }=0 &\text{on }B_r'.
        \end{cases}
    \end{align}
\end{proposition}

\begin{proof}
First, let us remark that by the regularity theory in \cite{SirTerVit21a,TerTorVit22}, the solution $h$ belongs to $C^\infty(B_r^+\cup B'_r)$ and satisfies the boundary condition
\begin{align}\label{eq:frozen-BC}
    \mean{A(0)\D h,\vec e_n }=0\quad\text{on }B_r'.
\end{align}
Let us remark that here the $C^2$ estimate in \cite[Proposition 4.4]{DonPha20}, which does not require symmetry of coefficients, is enough to run the linear approximation described below.
We define a linear polynomial $l(x):=h(0)+\mean{\D h(0), x}$. Then \eqref{eq:frozen-BC} readily gives the boundary condition in \eqref{eqlin}. \eqref{eq:frozen-BC} also implies the equation
\begin{align*}
    \div(x_n^aA(0)\D l)&=\div(x_n^aA(0)\D h(0))=\partial_{n}(x_n^a\mean{A(0)\D h(0),\vec e_n})=0.
\end{align*}
Next, we obtain by Taylor's theorem with remainder and rescaling the derivative estimates for $h$ (see e.g. \cite{TerTorVit22}*{Lemma~2.7})
$$
|h(x)-l(x)|^2\le C\|D^2h\|_{L^\infty(B^+_{r/2})}^2|x|^4\le \frac1{r^{n+4+a}}\left(\int_{B_r^+}x_n^ah^2\right)|x|^4,\quad x\in B^+_\rho.
$$
This readily yields the first estimate of the proposition
$$
\int_{B_\rho^+}x_n^a|h-l|^2\le \frac{C}{r^{n+4+a}}\left(\int_{B_r^+}x_n^ah^2\right)\int_{B_\rho^+}x_n^a|x|^4\le C(\rho/r)^{n+4+a}\int_{B_r^+}x_n^ah^2.
$$
For the second one, we observe that $l(0)=h(0)$ and $\D l\equiv\D h(0)$ and use the interior derivative estimates again to obtain
\begin{equation*}
|l(0)|^2+r^2|\D l|^2\le \|h\|_{L^\infty(B^+_{r/2})}^2+r^2\|\D h\|_{L^\infty(B^+_{r/2})}^2\le \frac{C}{r^{n+a}}\int_{B_r^+}x_n^ah^2.\qedhere
\end{equation*}
\end{proof}

\begin{proposition}
\label{prop:Dini-sol-freeze-coeffi-est}
Suppose $w\in H^{1,a}(B_r^+)$ is a solution of \eqref{evenLa3} in $B_r^+$. Let $h$ be a solution of \begin{align*}
\div(x_n^aA(0)\D h)=0\quad\text{in }B_{r/2}^+,
\end{align*}
with $h-w\in H^{1,a}_0(B_{r/2}^+\cup B'_{r/2})$. Then 
\begin{align*}
\int_{B_{r/2}^+}x_n^ah^2&\le C\int_{B_r^+}x_n^a(r^2|\bff|^2+w^2),\\
\int_{B_{r/2}^+}x_n^a|h-w|^2&\le C\int_{B_r^+}x_n^a(r^2|\bff|^2+L_1^2\sigma(\delta_0 r)^2w^2).
\end{align*}
\end{proposition}

To prove Proposition~\ref{prop:Dini-sol-freeze-coeffi-est}, we need the following auxiliary lemmas.

\begin{lemma}[Caccioppoli Inequality]\label{lem:Caccioppoli}
If $w\in H^{1,a}(B_r^+)$ is a solution of \eqref{evenLa3}, then
$$
\int_{B_{r/2}^+}x_n^a|\D w|^2\le C\int_{B_r^+}x_n^a\left(|\bff|^2+\frac{w^2}{r^2}\right).
$$    
\end{lemma}

\begin{proof}
Let $\vp:\R^n\to\R$ be a smooth function such that $0\le\vp\le1$ and $|\D\vp|\le \frac4r$ in $B_r$, $\vp=1$ in $B_{r/2}$, and $\vp=0$ in $\R^n\setminus B_r$. Since $-\div(x_n^aA\D w)=\div(x_n^a\bff)$ and $\vp^2w$ is a valid test function in $B_r^+$, we have 
\begin{align}
    \label{eq:sol-weak-form}
    \int_{B_r^+}x_n^a\mean{A\D w,\D(\vp^2w)}=-\int_{B_r^+}x_n^a\mean{\bff,\D(\vp^2w)}.
\end{align}
A simple algebra gives $$
\mean{A\D w,\D(\vp^2w)}=\mean{A\D(\vp w),\D(\vp w)}-w^2\mean{A\D\vp,\D\vp}.
$$
This, together with the uniform ellipticity of $A$ and $|\D \vp|\le 4/r$, implies the lower-bound for the left-hand side of \eqref{eq:sol-weak-form} \begin{align*}
\int_{B_r^+}x_n^a\mean{A\D w,\D(\vp^2w)}&=\int_{B_r^+}x_n^a\mean{A\D(\vp w),\D(\vp w)}-\int_{B_r^+}x_n^aw^2\mean{A\D\vp,\D\vp}\\
&\ge\frac1\Lambda\int_{B_r^+}x_n^a|\D(\vp w)|^2-\Lambda\int_{B_r^+}x_n^aw^2|\D\vp|^2\\
&\ge\frac1\Lambda\int_{B_r^+}x_n^a|\D(\vp w)|^2-C\int_{B_r^+}x_n^a\frac{w^2}{r^2}.
\end{align*}
We next estimate the right-hand side of \eqref{eq:sol-weak-form} by using Young's inequality
\begin{align*}
-\int_{B_r^+}x_n^a\mean{\bff,\D(\vp\cdot\vp w)}&=-\int_{B_r^+}x_n^a(\vp \mean{\bff,\D(\vp w)}+\vp w\mean{\bff,\D\vp})\\
&\le \int_{B_r^+}x_n^a\left(\frac{1}{2\Lambda}|\D(\vp w)|^2+C|\bff|^2\vp^2+w^2|\D\vp|^2\right)\\
&\le \frac{1}{2\Lambda}\int_{B_r^+}x_n^a|\D(\vp w)|^2+C\int_{B_r^+}x_n^a\left(|\bff|^2+\frac{w^2}{r^2}\right).
\end{align*}
Combining the above two estimates with \eqref{eq:sol-weak-form} (and using $\vp=1$ in $B_{r/2}^+$) yields
\begin{equation*}
\int_{B_{r/2}^+}x_n^a|\D w|^2\le \int_{B_r^+}x_n^a|\D(\vp w)|^2\le C\int_{B_r^+}x_n^a\left(|\bff|^2+\frac{w^2}{r^2}\right). \qedhere
\end{equation*}
\end{proof}

\begin{lemma}\label{lem:w-h-est}
For $w$ and $h$ as in Proposition~\ref{prop:Dini-sol-freeze-coeffi-est}, we have
\begin{align*}
    \int_{B_{r/2}^+}x_n^a|\D h|^2&\le C\int_{B_{r/2}^+}x_n^a|\D w|^2,\\
    \int_{B_{r/2}^+}x_n^a|\D(w-h)|^2&\le C\int_{B_r^+}x_n^a|\bff-\bg|^2,
\end{align*}
where $\bg:=(A(0)-A)\D h$.
\end{lemma}

\begin{proof}
As $\div(x_n^aA(0)\D h)=0$ in $B_r^+$ and $w-h\in H^{1,a}_0(B_{r/2}^+\cup B'_{r/2})$, we have 
$$
\int_{B_{r/2}^+}x_n^a\mean{A(0)\D h,\D(w-h)}=0.
$$
By making use of Young's inequality and the uniform ellipticity of $A$, we can easily get the first inequality of the lemma. For the second one, we observe that 
\begin{align*}
    -\div(x_n^aA\D(w-h))&=\div(x_n^a\bff)+\div(x_n^aA\D h)=\div(x_n^a\bff)+\div(x_n^a(A-A(0))\D h)\\
    &=\div(x_n^a(\bff-\bg)).
\end{align*}
This implies
$$
\int_{B^+_{r/2}}x_n^a\mean{A\D(w-h),\D(w-h)}=\int_{B^+_{r/2}}x_n^a\mean{\bg-\bff,\D(w-h)}.
$$
By using the uniform ellipticity of $A$ and applying Young's inequality, we can obtain the second estimate.    
\end{proof}

We are now ready to prove Proposition~\ref{prop:Dini-sol-freeze-coeffi-est}

\begin{proof}[Proof of Proposition~\ref{prop:Dini-sol-freeze-coeffi-est}]
We first observe that 
$$
-\div(x_n^aA\D h)=\div(x_n^a(A(0)-A)\D h)=\div(x_n^a\bg),\quad \text{where }\bg=(A(0)-A)\D h.
$$
By applying Lemma~\ref{lem:w-h-est} and Lemma~\ref{lem:Caccioppoli}, we get
$$
\int_{B_{r/2}^+}x_n^a|\D h|^2\le C\int_{B_{r/2}^+}x_n^a|\D w|^2\le C\int_{B_r^+}x_n^a\left(|\bff|^2+\frac{w^2}{r^2}\right).
$$
This gives the estimate for $\bg$
$$
\int_{B_{r/2}^+}x_n^a|\bg|^2\le CL_1^2(\sigma(\de_0 r))^2\int_{B_{r/2}^+}x_n^a|\D h|^2\le CL_1^2(\sigma(\de_0 r))^2\int_{B_r^+}x_n^a\left(|\bff|^2+\frac{w^2}{r^2}\right).
$$
Since $w-h$ satisfies
$$
-\div(x_n^aA\D(w-h))=\div(x_n^a(\bff-\bg))\quad\text{in }B_{r/2}^+,
$$
we have 
\begin{align*}
\int_{B_{r/2}^+}x_n^a|w-h|^2&\le Cr^2\int_{B_{r/2}^+}x_n^a|\D(w-h)|^2\le Cr^2\int_{B_r^+}x_n^a|\bff-\bg|^2\\
&\le Cr^2\int_{B_r^+}x_n^a(|\bff|^2+|\bg|^2)\le C\int_{B_r^+}x_n^a(r^2|\bff|^2+L_1^2\sigma(\de_0 r)^2w^2).
\end{align*}
Here, we used Lemma~\ref{lem:Poincare} in the first inequality and Lemma~\ref{lem:w-h-est} in the second inequality. The last step follows from the estimate for the weighted $L^2$-norm of $\bg$ obtained above. This is the second estimate in the lemma, and it readily gives the first one
\begin{equation*}
\int_{B_{r/2}^+}x_n^ah^2\le 2\int_{B_{r/2}^+}x_n^a(|w-h|^2+w^2)\le C\int_{B_r^+}x_n^a(r^2|\bff|^2+w^2).\qedhere
\end{equation*}
\end{proof}

\begin{proposition}\label{prop:Dini-poly-approx}
For a solution $w$ of \eqref{evenLa3} in $B_1^+$, we denote $$
\sup_{x\in B_1^+}\frac{|\bff(x)-\bff(0)|}{\sigma(|x|)}=\rchi,\quad \int_{B_1^+}x_n^aw^2=\mu^2,\quad |\bff(0)|=\tau.
$$
If $\de_0$ is sufficiently small universal, then there exists an affine function $P$ such that $$
\|P\|_{C^1(B_1^+)}^2+\frac1{r^{n+2+a}\mathcal{J}_\sigma(r)^2}\int_{B_r^+}x_n^a|w-P|^2\le C(\chi^2+\mu^2+\tau^2),\quad 0<r<r_0,
$$
for some universal constants $r_0>0$ and $C>0$. In particular,
\begin{equation}\label{BCP}
\mean{A(0)\D P+\bff(0),\vec e_n}=0.
\end{equation}
Moreover, there exists $\bff_\infty$ with $\bff_\infty(0)=0$ such that $$
-\div(x_n^aA\D(w-P))=\div(x_n^a\bff_\infty),\quad \sup_{x\in B_1^+}\frac{|\bff_\infty(x)|}{\sigma(|x|)}\le C(\chi+\mu+\tau).
$$
\end{proposition}

\begin{proof} We split our proof into two cases $$
\text{either $\tau=0$ or $\tau\neq0$.  }
$$
\emph{Case 1.} We first consider the case $|\bff(0)|=\tau=0$, and further divide its proof into several steps.

\medskip\noindent\emph{Step 1.} We inductively define a sequence of functions $\{w_k\}_{k=0}^\infty$
\begin{align*}
\begin{cases}
w_0:=w,\\
w_{k+1}=w_k-l_k,\quad k\ge0,
\end{cases}
\end{align*}
where $l_k's$ are linear polynomials defined as follows: for a small universal constant $\la>0$ to be chosen later, let $h_k$, $k\ge0$, be a solution of 
\begin{align*}
    \begin{cases}
        \div(x_n^aA(0)\D h_k)=0&\text{in }B^+_{\la^k/2},\\
        h_k=w_k&\text{on }\partial^+ B_{\la^k/2},
    \end{cases}
\end{align*}
and let $l_k$ be the linearization of $h_k$ as in Proposition~\ref{prop:freeze-coeffi-poly-approx}; that is,
\begin{equation*}
l_k(x)=h_k(0)+\mean{\nabla h_k(0),x}.
\end{equation*}
We set
$$
\begin{cases}
   \bff_0:=\bff,\\
    \bff_{k+1}:=\bff_k+(A-A(0))\D l_k,\quad k\ge0.
\end{cases}
$$
Note that $\bff_k(0)=0$. We claim that \begin{align}
\label{eq:Dini-induc}
-\div(x_n^aA\D w_k)=\div(x_n^a\bff_k).
\end{align}
To prove it, we use induction on $k$. It is clearly true when $k=0$. We now assume that \eqref{eq:Dini-induc} holds true for $k$, and prove it for $k+1$. Using Proposition~\ref{prop:freeze-coeffi-poly-approx}, we get
\begin{align*}
-\div(x_n^aA\D w_{k+1})&=-\div(x_n^aA\D w_k)+\div(x_n^aA\D l_k)\\
&=\div(x_n^a\bff_k)+\div(x_n^a(A-A(0))\D l_k)=\div(x_n^a\bff_{k+1}).
\end{align*}

\medskip\noindent\emph{Step 2.} We define sequences $\rchi_k\ge0$ and $\mu_k\ge0$ as
$$
\rchi_k:=\sup_{x\in B_{\la^k}^+}\frac{|\bff_k(x)|}{\sigma(|x|)},\quad \mu_k^2:=\frac1{\la^{k(n+2+a)}\sigma(\la^k)^2}\int_{B_{\la^k}^+}x_n^aw_k^2.
$$
We claim that 
\begin{align}
    \label{eq:linear-poly-est} &|l_k(0)|^2+\la^{2k}|\D l_k|^2\le C\la^{2k}\sigma(\la^k)^2(\rchi_k^2+\mu_k^2),\\
    \label{eq:sol-linear-poly-diff-est} &\frac1{\la^{k(n+2+a)}\sigma(\la^k)^2}\int_{B^+_{\la^{k+1}}}x_n^a|w_k-l_k|^2\le C(\rchi_k^2+\la^{n+4+a}\mu_k^2).
\end{align}
Indeed, by applying Proposition~\ref{prop:freeze-coeffi-poly-approx} and Proposition~\ref{prop:Dini-sol-freeze-coeffi-est}, we can obtain \eqref{eq:linear-poly-est}; that is,
\begin{align}
    \label{eq:linear-poly-est-comput}
    \begin{split}
    |l_k(0)|^2+\la^{2k}|\D l_k|^2&\le \frac{C}{\la^{k(n+a)}}\int_{B^+_{\la^k/2}}x_n^ah_k^2\le\frac{C}{\la^{k(n+a)}}\int_{B^+_{\la^k}}x_n^a(\la^{2k}|\bff_k|^2+w_k^2)\\
    &\le C\la^{2k}\sigma(\la^k)^2(\rchi_k^2+\mu_k^2).
\end{split}\end{align}
For \eqref{eq:sol-linear-poly-diff-est}, we observe that
$$
\int_{B^+_{\la^{k+1}}}x_n^a|w_k-l_k|^2\le 2\int_{B^+_{\la^{k+1}}}x_n^a|w_k-h_k|^2+2\int_{B_{\la^{k+1}}^+}x_n^a|h_k-l_k|^2=:I+II.
$$
The estimate for $I$ follows by applying Proposition~\ref{prop:Dini-sol-freeze-coeffi-est}
\begin{align*}
    \int_{B^+_{\la^{k+1}}}x_n^a|w_k-h_k|^2&\le \int_{B^+_{\la^k/2}}x_n^a|w_k-h_k|^2\le C\int_{B^+_{\la^k}}x_n^a\left(\la^{2k}|\bff_k|^2+L_1^2\sigma(\delta_0\la^k)^2w_k^2\right)\\
    &\le C\la^{k(n+2+a)}\sigma(\la^k)^2\left(\rchi_k^2+L_1^2\sigma(\delta_0\la^k)^2\mu_k^2\right).
\end{align*}
For $II$, we use Proposition~\ref{prop:freeze-coeffi-poly-approx} to get
\begin{align*}
    \int_{B^+_{\la^{k+1}}}x_n^a|h_k-l_k|^2\le C\la^{n+4+a}\int_{B^+_{\la^k/2}}x_n^ah_k^2\le C\la^{k(n+2+a)+n+4+a}\sigma(\la^k)^2(\rchi_k^2+\mu_k^2),
\end{align*}
where in the second inequality we have used the computation made in \eqref{eq:linear-poly-est-comput}. If $\delta_0=\delta_0(\la,\sigma, L_1,a)$ is small so that $L_1^2\sigma(\delta_0\la^k)^2\le L_1^2\sigma(\delta_0)^2\le \la^{n+4+a}$, then we conclude \eqref{eq:sol-linear-poly-diff-est}
$$
\int_{B^+_{\la^{k+1}}}x_n^a|w_k-l_k|^2\le I+II\le C\la^{k(n+2+a)}\sigma(\la^k)^2(\rchi_k^2+\la^{n+4+a}\mu_k^2).
$$

\medskip\noindent\emph{Step 3.} We show in this step that $\rchi_k$ and $\mu_k$ are uniformly bounded. To this aim, we use \eqref{eq:linear-poly-est} to get
\begin{align*}
    \rchi_{k+1}&=\sup_{x\in B_{\la^{k+1}}^+}\frac{|\bff_{k+1}(x)|}{\sigma(|x|)}\le \sup_{x\in B_{\la^k}^+}\frac{|\bff_{k}(x)|}{\sigma(|x|)}+\sup_{x\in B_{\la^{k+1}}^+}\frac{|(A(x)-I)\D l_k|}{\sigma(|x|)}\\
    &\le \rchi_k+C_1\sigma(\la^k)(\rchi_k+\mu_k)=(1+C_1\sigma(\la^k))\rchi_k+C_1\sigma(\la^k)\mu_k,
\end{align*}
for some universal constant $C_1$ independent of $\la$.
On the other hand, by exploiting \eqref{eq:sol-linear-poly-diff-est}, we obtain
\begin{align*}
    \mu_{k+1}^2&=\frac1{\la^{(k+1)(n+2+a)}\sigma(\la^{k+1})^2}\int_{B_{\la^{k+1}}^+}x_n^aw_{k+1}^2\\
    &=\frac{\sigma(\la^k)^2}{\la^{n+2+a}\sigma(\la^{k+1})^2}\left(\frac1{\la^{k(n+2+a)}\sigma(\la^k)^2}\int_{B_{\la^{k+1}}^+}x_n^a|w_k-l_k|^2\right)\\
    &\le \frac{\sigma(\la^k)^2}{\la^{n+2+a}\sigma(\la^{k+1})^2}(C\rchi_k^2+C\la^{n+4+a}\mu_k^2),
\end{align*}
and thus 
$$
\mu_{k+1}\le C_2\frac{\sigma(\la^k)}{\la^{\frac {n+a}2+1}\sigma(\la^{k+1})}\rchi_k+C_2\frac{\la\sigma(\la^k)}{\sigma(\la^{k+1})}\mu_k,
$$
for another universal constant $C_2$ independent of $\la$. Since $\frac{\sigma(t)}{t^\al}$ is nonincreasing for some $0<\al<1$, we have $\frac{\sigma(\la^k)}{\sigma(\la^{k+1})}\le \frac1{\la^\al}$. Fixing $\la<1/4$ small universal satisfying $C_2\la^{1-\al}\le 1/2$, we find
$$
\begin{cases}
    \rchi_{k+1}\le (1+C_1\sigma(\la^k))\rchi_k+C_1\sigma(\la^k)\mu_k,\\
    \mu_{k+1}\le \frac{C_2}{\la^{\frac {n+a}2+1+\al}}\rchi_k+\frac12\mu_k.
\end{cases}
$$
Calling $\tilde\mu_k:=\frac{\la^{\frac{n+a}2+1+\al}}{2C_2}\mu_k$, we have 
$$
\begin{cases}
    \rchi_{k+1}\le (1+C_1\sigma(\la^k))\rchi_k+\frac{2C_1C_2\sigma(\la^k)}{\la^{\frac{n+a}2+1+\al}}\tilde\mu_k,\\
    \tilde\mu_{k+1}\le 1/2(\rchi_k+\tilde\mu_k).
\end{cases}
$$
It follows that for a universal constant $C_3:=\frac{2C_1C_2}{\la^{\frac{n+a}2+1+\al}}$,
$$
\begin{cases}
    \rchi_{k+1}\le (1+C_1\sigma(\la^k))\rchi_k+C_3\sigma(\la^k)\tilde\mu_k,\\
    \tilde\mu_{k+1}\le 1/2(\rchi_k+\tilde\mu_k).
\end{cases}
$$
By using an induction on $k$, one can easily see that $$
\rchi_k\le \tilde C_k(\rchi_0+\tilde\mu_0),\quad \tilde\mu_k\le \tilde C_k(\rchi_0+\tilde\mu_0),
$$
where $\tilde C_k:=\Pi_{i=0}^k(1+(C_1+C_3)\sigma(\la^i))$. Using the fact that $\log(1+x)\le x$ for $x>0$ and $(8)$ in Remark~\ref{remarkJalpha}, we get 
\begin{align*}
    \log \tilde C_k&=\sum_{i=0}^k\log(1+(C_1+C_3)\sigma(\la^i))\le (C_1+C_3)\sum_{i=0}^k\sigma(\la^i)\\
    &\le (C_1+C_3)\sum_{i=0}^\infty\sigma(\la^i)\le (C_1+C_3)\frac{\cJ_\sigma(1)}{1-\la}.    
\end{align*}
Thus, $\tilde C_k$ is uniformly bounded, and hence (using $\tilde \mu_k=C\mu_k$)
$$
\rchi_k\le C(\rchi+\mu),\quad \mu_k\le C(\rchi+\mu).
$$

\medskip\noindent\emph{Step 4.} We write $P_k(x):=\sum_{i=0}^kl_i(x)$. Note that $w_k=w-P_{k-1}$, $k\ge1$. From \eqref{eq:linear-poly-est}, $$
|l_i(0)|+|\D l_i|\le C\sigma(\la^i)(\rchi_i+\mu_i)\le C(\rchi+\mu)\sigma(\la^i),
$$
thus $$
|P_k(0)|+|\D P_k|\le C(\rchi+\mu)\sum_{i=0}^k\sigma(\la^i)<\infty,
$$
and hence $|P_k(0)|+|\D P_k|$ is uniformly bounded. This gives that over a subsequence $P_k\to P$ in $C^1(\overline{B_1^+})$ for some affine function $P(x)$ with $$\|P\|_{C^1(B_1^+)}\le C(\rchi+\mu).$$
Moreover, we have by Proposition~\ref{prop:freeze-coeffi-poly-approx}
$$
\mean{A(0)\D P_k,\vec e_n}=\sum_{i=0}^k\mean{A(0)\D l_i,\vec e_n}=0,
$$
and thus 
\begin{align*}
\mean{A(0)\D P,\vec e_n}=0.
\end{align*}
As $\bff(0)=0$, this gives \eqref{BCP}.

To estimate the weighted $L^2$-distance between $w$ and $P$, let $r\in(0,\la)$ be given. Then we can find $k\ge1$ such that $\la^{k+1}\le r<\la^k$. Note that
\begin{align}\label{eq:sol-linear-poly-diff-est-2}
\frac1{r^{n+2+a}\mathcal{J}_\sigma(r)^2}\int_{B_r^+}x_n^a|w-P|^2\le \frac1{\la^{(k+1)(n+2+a)}\cJ_\sigma(\la^{k+1})^2}\int_{B_{\la^k}^+}x_n^a|w-P|^2.
\end{align}
To bound $\int_{B_{\la^k}^+}x_n^a|w-P|^2$, we observe that 
\begin{align*}
    |w-P|&\le |w_k|+|w-w_k-P|=|w_k|+|P_{k-1}-P|\le |w_k|+\sum_{i=k}^\infty|l_i|,
\end{align*}
and thus
\begin{align*}
    \left(\int_{B_{\la^k}^+}x_n^a|w-P|^2\right)^{1/2}&\le \left(\int_{B_{\la^k}^+}x_n^aw_k^2\right)^{1/2}+\sum_{i=k}^\infty\left(\int_{B_{\la^k}^+}x_n^al_i^2\right)^{1/2}\\
    &\le C\la^{k(\frac{n+a}2+1)}\sigma(\la^k)\mu_k+\sum_{i=k}^\infty\left(\int_{B_{\la^k}^+}x_n^al_i^2\right)^{1/2}.
\end{align*}
We use \eqref{eq:linear-poly-est} to get that for every $i\ge k$
\begin{align*}
    \sup_{B_{\la^k}^+}|l_i|&\le |l_i(0)|+\sup_{B^+_{\la^k}}|l_i-l_i(0)|\le |l_i(0)|+\la^k\sup_{B^+_{\la^k}}|\D l_i|\\
    &\le C\la^k\sigma(\la^i)(\rchi_i+\mu_i)\le C\la^k\sigma(\la^i)(\rchi+\mu),
\end{align*}
which implies 
$$
\left(\int_{B^+_{\la^k}}x_n^al_i^2\right)^{1/2}\le C\la^{\frac{k(n+a)}2}\sup_{B^+_{\la^k}}|l_i|\le C\la^{k(\frac {n+a}2+1)}\sigma(\la^i)(\rchi+\mu).
$$
Thus,
\begin{align*}   
    \left(\int_{B^+_{\la^k}}x_n^a|w-P|^2\right)^{1/2}&\le C\la^{k(\frac{n+a}2+1)}(\rchi+\mu)\left(\sigma(\la^k)+\sum_{i=k}^\infty\sigma(\la^i)\right)\\
    &\le C\la^{k(\frac{n+a}2+1)}(\rchi+\mu)\left(\sigma(\la^k)+\cJ_\sigma(\la^{k})\right)\\
    &\le C\la^{k(\frac{n+a}2+1)}(\rchi+\mu)\cJ_\sigma(\la^{k}),
\end{align*}
where we have used $(8)$ in Remark~2.1 in the second inequality. Combining this estimate with \eqref{eq:sol-linear-poly-diff-est-2} yields
\begin{align*}
    \frac1{r^{n+2+a}\cJ_\sigma(r)^2}\int_{B_r^+}x_n^a|w-P|^2&\le \frac1{\la^{(k+1)(n+2+a)}\cJ_\sigma(\la^{k+1})^2}\int_{B_{\la^k}^+}x_n^a|w-P|^2\\
    &\le \frac{C\la^{k(n+2+a)}\cJ_\sigma(\la^{k})^2}{\la^{(k+1)(n+2+a)}\cJ_{\sigma}(\la^{k+1})^2}(\rchi^2+\mu^2)\\
    &\le C(\rchi^2+\mu^2),
\end{align*}
where in the last inequality we have used that $\frac{\cJ_\sigma(r)}r$ is nonincreasing (so $\frac{\cJ_\sigma(\la^{k})}{\cJ_\sigma(\la^{k+1})}\le 1/\la$) and $\la$ is a universal constant.

Next, we observe that since $|\bff_{k+1}-\bff_k|=|(A(0)-A)\D l_k|\le C(\rchi+\mu)\sigma(\la^k)$ in $B_1^+$ and $\sum_{i=k}^\infty\sigma(\la^i)\le C\cJ_\sigma(\la^{k-1})\to0$ as $k\to\infty$, $\bff_k\to\bff_\infty$ uniformly in $B_1^+$. Note that $\bff_\infty(0)=0$ and that
\begin{align*}
    \sup_{x\in B_1^+}\frac{|\bff_\infty(x)|}{\sigma(|x|)}&\le \sup_{x\in B_1^+}\frac{|\bff(x)|}{\sigma(|x|)}+\sum_{k=0}^\infty\left(\sup_{x\in B_1^+}\frac{|\bff_{k+1}(x)-\bff_k(x)|}{\sigma(|x|)}\right)\\
    &\le \rchi+\sum_{k=0}^\infty\left(\sup_{x\in B_1^+}\frac{|(A(x)-A(0))\D l_k|}{\sigma(|x|)}\right)\\
    &\le \rchi+\sum_{k=0}^\infty C(\rchi+\mu)\sigma(\la^k)\le C(\rchi+\mu).
\end{align*}
We note $-\div(x_n^aA\D(w-P_k))=-\div(x_n^aA\D w_{k+1})=\div(x_n^a\bff_{k+1})$ and pass to the limit to obtain $$
-\div(x_n^aA\D(w-P))=\div(x_n^a\bff_\infty).
$$

\medskip\noindent\emph{Case 2.} We now consider the case $\tau=|\bff(0)|\neq0$. We set $$
\hat w(x):=w(x)+\mean{A^{-1}(0)\bff(0),x},\quad \hat \bff(x):=\bff(x)-A(x)A^{-1}(0)\bff(0).
$$
Then $\hat \bff(0)=0$ and 
\begin{align*}
-\div(x_n^aA\D \hat w)
&=-\div(x_n^aA\D w)-\div(x_n^aA\D\mean{A^{-1}(0)\bff(0),x})\\
&=\div(x_n^a\bff)-\div(x_n^aAA^{-1}(0)\bff(0))\\
&=\div(x_n^a\hat \bff).
\end{align*}
We also observe that 
\begin{align*}
    \hat \rchi&:=\sup_{x\in B_1^+}\frac{|\hat \bff(x)-\hat \bff(0)|}{\sigma(|x|)}\le \sup_{x\in B_1^+}\frac{|\bff(x)-\bff(0)|}{\sigma(|x|)}+\sup_{x\in B_1^+}\frac{|(A(x)-A(0))A^{-1}(0)\bff(0)|}{\sigma(|x|)}\\
    &\le \rchi+C\tau
\end{align*}
and 
\begin{align*}
    \hat \mu^2:=\int_{B_1^+}x_n^a\hat w^2\le 2\int_{B_1^+}x_n^aw^2+2\int_{B_1^+}x_n^a\mean{A^{-1}(0)\bff(0),x}^2\le C(\mu^2+\tau^2). 
\end{align*}
Since $\hat \bff(0)=0$, we can apply the result in the previous case to obtain the following: there exists a linear polynomial $\hat P$ with $\mean{A(0)\D\hat P,\vec e_n}=0$ such that
\begin{align*}
    \|\hat P\|_{C^1(B_1^+)}^2+\frac1{r^{n+2+a}\cJ_\sigma(r)^2}\int_{B_r^+}x_n^a|\hat w-\hat P|^2\le C(\hat \rchi^2+\hat \mu^2)\le C(\rchi^2+\mu^2+\tau^2),
\end{align*}
and a field $\hat\bff_\infty$ with $\hat\bff_\infty(0)=0$ such that 
$$
-\div(x_n^aA\D(\hat w-\hat P))=\div(x_n^a\hat\bff_\infty),\qquad \sup_{x\in B^+_1}\frac{|\hat\bff_\infty(x)|}{\sigma(|x|)}\le C(\hat\rchi+\hat\mu)\le C(\rchi+\mu+\tau).
$$
Then $w$ satisfies the statement of Proposition~\ref{prop:Dini-poly-approx} with $P(x):=\hat P(x)-\mean{A^{-1}(0)\bff(0),x}$ and $\bff_\infty(x):=\hat\bff_\infty(x)$. In particular, the boundary condition \eqref{BCP} follows from the equality $\mean{A(0)\D\hat P,\vec e_n}=0$
$$
\mean{A(0)\D P,\vec e_n}=\mean{A(0)\D\hat P,\vec e_n}-\mean{A(0)\D(\mean{A^{-1}(0)\bff(0),x},\vec e_n}=-\mean{\bff(0),\vec e_n}.
$$
This completes the proof. 
\end{proof}

\begin{theorem}
    \label{thm:Dini-ratio-reg}
    Suppose $w$ solves \eqref{evenLa3} in $B_1^+$ with $\bff\in C^{0,\sigma}(B_1^+)$ and small $\delta_0$ universal. Then there exists $\rho_0>0$ universal such that $w\in C^{1,\cH_{\sigma+\tilde\sigma}}(B_{\rho_0}^+)$ with 
    $$
    \|w\|_{C^{1,\cH_{\sigma+\tilde\sigma}}(B_{\rho_0}^+)}\le C\left(\|w\|_{L^2(B_1^+,x_n^adx)}+\|\bff\|_{C^{0,\sigma}(B_1^+)}\right).
    $$
 Moreover, $w$ satisfies
    \begin{align}\label{eq:w-BC}
        \mean{A\nabla w+\bff,\vec e_n}=0\quad\text{on }B'_{\rho_0}.
    \end{align}
\end{theorem}

\begin{proof}
\emph{Step 1.}
Let $P$, $\bff_\infty$, $r_0$ be as in Proposition~\ref{prop:Dini-poly-approx}. For $0<r<r_0/3$ we consider a scaling 
$$
x(y)=x_r(y):=ry+\frac{11}{10}r\vec e_n,\quad y\in B_1.
$$
If we set $w_r(y):=(w-P)(x(y))$, then $w_r(y)$ solves the uniformly elliptic equation 
$$
-\div\left(\left(\frac{11}{10}+y_n\right)^aA(x(y))\D w_r(y)\right)=\div\left(\left(\frac{11}{10}+y_n\right)^ar\bff_\infty(x(y))\right),\quad y\in B_1.
$$
Indeed, if we write $x(y)=(x_1(y),\cdots,x_n(y))$, then $x_n(y)=r\left(y_n+\frac{11}{10}\right)$, thus 
\begin{align*}
    -\div\left(\left(\frac{11}{10}+y_n\right)^aA(x(y))\D w_r(y)\right)&=-\div\left(\frac{x_n(y)^a}{r^a}A(x(y))(r\D(w-P)(x(y)))\right)\\
    &=\div\left(\frac{x_n(y)^a}{r^{a-1}}\bff_\infty(x(y))\right)\\
    &=\div\left(\left(\frac{11}{10}+y_n\right)^ar\bff_\infty(x(y))\right).
\end{align*}
Writing 
$$A_r(y):=\left(\frac{11}{10}+y_n\right)^aA(x(y))\quad\text{and}\quad
\bg_r(y):=\left(\frac{11}{10}+y_n\right)^ar\bff_\infty(x(y)),\quad y\in B_1,
$$
we have 
$$
-\div(A_r\D w_r)=\div\bg_r\quad\text{in }B_1.
$$
We apply the results in \cite{DonEscKim18} (in particular, Lemma~2.11 and the equation (2.36)) to get 
\begin{align*}
   \|\D w_r\|_{L^\infty(B_{1/2})}\le C\|\D w_r\|_{L^2(B_{3/4})}+C\int_0^1\frac{\hat{\omega}_{\bg_r}(t)}t\,dt
   \end{align*}
and
   \begin{align*}
   |\D w_r(y^1)-\D w_r(y^2)|&\le C\|\D w_r\|_{L^2(B_{3/4})}|y^1-y^2|^\al+\omega_{\bg_r}^*(|y^1-y^2|)\\
   &\qquad+C\left(\|\D w_r\|_{L^2(B_{3/4})}+\int_0^1\frac{\hat{\omega}_{\bg_r}(t)}t\,dt\right)\omega_{A_r}^*(|y^1-y^2|),
\end{align*}
for any $y^1,\,y^2\in B_{1/2}$. As observed in the proof of Corollary~\ref{cor:reg}, $\hat\omega_{\bg_r}$, $\omega_{\bg_r}^*$ and $\omega_{A_r}^*$ can be redefined as 
$$
\hat\omega_{\bg_r}(t)=\tilde\omega_{\bg_r}(t),\quad \omega^*_{\bg_r}(t)=\cJ_{\tilde\omega_{\bg_r}}(t),\quad \omega^*_{A_r}(t)=\cJ_{\tilde\omega_{A_r}}(t).
$$
Then, the above estimates are equivalent to
\begin{align}\label{eq:grad-w-p-infty}
    \|\D(w-P)\|_{L^\infty(B_{r/2}(\frac{11}{10}r\vec e_n))}\le \frac{C}{r}\left(\|\D w_r\|_{L^2(B_{3/4})}+\cJ_{\tilde\omega_{\bg_r}}(1)\right),
\end{align}
and for any $x^1,\,x^2\in B_{r/2}(\frac{11}{10}r\vec e_n)$ 
\begin{align}\label{eq:grad-w-mod}\begin{split}
    |\D w(x^1)-\D w(x^2)|&\le \frac{C}{r}\|\D w_r\|_{L^2(B_{3/4})}\left|\frac{x^1-x^2}r\right|^\al+\frac1r\cJ_{\tilde\omega_{\bg_r}}\left(\frac{|x^1-x^2|}r\right)\\
    &\qquad+\frac{C}{r}\left(\|\D w_r\|_{L^2(B_{3/4})}+\cJ_{\tilde\omega_{\bg_r}}(1)\right)\cJ_{\tilde\omega_{A_r}}\left(\frac{|x^1-x^2|}r\right).
\end{split}\end{align}
We remark that instead of the previous boundary estimate in \cite{DonEscKim18}, the interior estimate in \cite{DonKim17} can be used and is easier to apply. However, we chose to exploit the one in \cite{DonEscKim18}, since those two results are simplified into the same estimate \eqref{eq:grad-w-p-infty}, \eqref{eq:grad-w-mod} under the $\al$-nonincreasing condition and the boundary estimate \cite{DonEscKim18} is valid for any $0<\al<1$ while the interior one \cite{DonKim17} works only for small $\al>0$.

\medskip\noindent \emph{Step 2.} In this step, we deal with the right-hand sides of \eqref{eq:grad-w-p-infty} and \eqref{eq:grad-w-mod}, i.e., we estimate $\|\D w_r\|_{L^2(B_{3/4})}$, $\cJ_{\tilde\omega_{\bg_r}}(1)$, $\cJ_{\tilde\omega_{\bg_r}}\left(\frac{|x^1-x^2|}r\right)$ and $\cJ_{\tilde\omega_{A_r}}\left(\frac{|x^1-x^2|}r\right)$.

For notational simplicity we denote by $C_{\bff,w}$ a constant (which may vary in each line) of the form 
$$
C_{\bff,w}=C(n,a,\Lambda,\al,\sigma)\left(\|w\|_{L^2(B_1^+,x_n^adx)}+\|\bff\|_{C^{0,\sigma}(B_1^+)}\right).
$$
Note that for $\rchi$, $\mu$, $\tau$ as in Proposition~\ref{prop:Dini-poly-approx}, $\rchi+\mu+\tau\le C_{\bff,w}$.

\medskip\noindent \emph{Step 2.1.} We first estimate $\|\D w_r\|_{L^2(B_{3/4})}$. By Proposition~\ref{prop:Dini-poly-approx}, we have
\begin{align*}
    \int_{B_1}|w_r(y)|^2\,dy&=\int_{B_1}|(w-P)(x(y))|^2\,dy\le \frac{C}{r^n}\int_{B_r(\frac{11}{10}r\vec e_n)}|w-P|^2\\
    &\le \frac{C}{r^{n+a}}\int_{B_r(\frac{11}{10}r\vec e_n)}x_n^a|w-P|^2\le\frac{C}{r^{n+a}}\int_{B_{3r}^+}x_n^a|w-P|^2\\
    &\le C_{\bff,w}^2r^2\cJ_\sigma(3r)^2\le C_{\bff,w}^2r^2\cJ_\sigma(r)^2,
\end{align*}
where in the third step we have used that $\frac1{10}r<x_n<\frac{21}{10}r$ in $B_r(\frac{11}{10}r\vec e_n)$, so $c(a)\le \frac{x_n^a}{r^a}\le C(a)$ in there for some positive constants $c(a)$ and $C(a)$ depending only on $a$.
Using Proposition~\ref{prop:Dini-poly-approx} again, together with $\bff_\infty(0)=0$, we see that for any $y\in B_1$
\begin{align*}
    |\bg_r(y)|&\le Cr|\bff_\infty(x(y))-\bff_\infty(0)|\le C_{\bff,w}r\sigma(|x(y)|)\\
    &\le C_{\bff,w}r\sigma(3r)\le C_{\bff,w}r\sigma(r).
\end{align*}
By applying Caccioppoli inequality and using the above estimates, we obtain
\begin{align*}
    \int_{B_{3/4}}|\D w_r(y)|^2\le C\int_{B_1}\left(|w_r(y)|^2+|\bg_r|^2\right)\le C_{\bff,w}^2r^2\cJ_\sigma(r)^2.
\end{align*}

\medskip\noindent \emph{Step 2.2.} In the following \emph{Step 2.2} - \emph{Step 2.4}, we estimate $\omega^*_{\bg_r}(t)$ and $\omega^*_{A_r}(t)$.

We first find moduli of continuity of $\bg_r$ and $A_r$. For this purpose, we claim that $\rho\sigma(r)\le 2\sigma(r\rho)$ for every $0<\rho<2$. Indeed, if $0<\rho<1$, then $\frac{\sigma(r)}{r}\le \frac{\sigma(r\rho)}{r\rho}$ by the monotonicity of $t\longmapsto \frac{\sigma(t)}t$, thus $\rho\sigma(r)\le \sigma(r\rho)$. On the other hand, if $1\le \rho<2$, then $\rho\sigma(r)\le 2\sigma(r)\le 2\sigma(r\rho)$.

Now, for any $y^1$, $y^2\in B_1$, we have 
\begin{align*}
    |\bg_r(y^1)-\bg_r(y^2)|&\le \left(\frac{11}{10}+y^1_n\right)^ar|\bff_\infty(x(y^1))-\bff_\infty(x(y^2))|\\
    &\qquad+\left|\left(\frac{11}{10}+y^1_n\right)^a-\left(\frac{11}{10}+y^2_n\right)^a\right|r|\bff_\infty(x(y^2))|\\
    &\le C_{\bff,w}r\sigma(|x(y^1)-x(y^2)|)+C|y^1-y^2|r|\bff_\infty(x(y^2))-\bff_\infty(0)|\\
    &\le C_{\bff,w}r\sigma(r|y^1-y^2|)+C_{\bff,w}|y^1-y^2|r\sigma(2r)\\
    &\le C_{\bff,w}r\sigma(r|y^1-y^2|)+C_{\bff,w}r\sigma(r)|y^1-y^2|\\
    &\le C_{\bff,w}r\sigma(r|y^1-y^2|).
\end{align*}
Similarly, for $y^1,\,y^2\in B_1$, 
\begin{align*}
    &|A_r(y^1)-A_r(y^2)|\\
    &\qquad\le \left(\frac{11}{10}+y^1_n\right)^a|A(x(y^1))-A(x(y^2))|+|A(x(y^2))|\left|\left(\frac{11}{10}+y^1_n\right)^a-\left(\frac{11}{10}+y^2_n\right)^a\right|\\
    &\qquad\le C_{\bff,w}\sigma(r|y^1-y^2|)+C_{\bff,w}|y^1-y^2|.
\end{align*}
Therefore, 
$$
\omega_{\bg_r}(t)\le \eta_{\bg_r}(t)\le C_{\bff,w}r\sigma(rt),\quad \omega_{A_r}(t)\le \eta_{A_r}(t)\le C_{\bff,w}(\sigma(rt)+t).
$$

\medskip\noindent \emph{Step 2.3.} We claim that 
\begin{align*}
    \tilde\omega_{\bg_r}(t)\le C_{\bff,w}r\left(\tilde\sigma(rt)+\sigma(r)t^\al\right),\qquad \cJ_{\tilde\omega_{\bg_r}}(t)\le C_{\bff,w}r\left(\cJ_{\tilde\sigma}(rt)+\sigma(r)t^\al\right).
\end{align*}
Indeed, given $t\in (0,1)$, take $N\in \mathbb{N}$ such that $\ka^{-N}t\le 1<\ka^{-(N+1)}t$. Then 
$$
\tilde\omega_{\bg_r}(t)=\sum_{i=1}^\infty\ka^{i\al}\omega_{\bg_r}(\ka^{-i}t)[\ka^{-i}t\le1]+\sum_{i=N+1}^\infty\ka^{i\al}\omega_{\bg_r}(1).
$$
Using $\ka^{N+1}<t$, we can compute
$$
\sum_{i=N+1}^\infty\ka^{i\al}\omega_{\bg_r}(1)\le C\omega_{\bg_r}(1)\ka^{(N+1)\al}\le C\omega_{\bg_r}(1)t^\al.
$$
Thus
\begin{align*}
    \tilde\omega_{\bg_r}(t)&\le \sum_{i=1}^\infty\ka^{i\al}\omega_{\bg_r}(\ka^{-i}t)[\ka^{-i}t\le1]+C\omega_{\bg_r}(1)t^\al\\
    &\le C_{\bff,w}r\sum_{i=1}^\infty\ka^{i\al}\sigma(\ka^{-i}rt)[\ka^{-i}rt\le1]+C_{\bff,w}r\sigma(r)t^\al\\
    &\le C_{\bff,w}r\tilde\sigma(rt)+C_{\bff,w}r\sigma(r)t^\al.
\end{align*}
This is the first estimate in the claim. The second estimate follows from the first one
\begin{align*}
    \cJ_{\tilde\omega_{\bg_r}}(t)&\le C_{\bff,w}r\int_0^t\frac{\tilde\sigma(rs)+\sigma(r)s^\al}s\,ds=C_{\bff,w}r\left(\int_0^{rt}\frac{\tilde\sigma(\rho)}\rho\,d\rho+\sigma(r)t^\al\right)\\
    &=C_{\bff,w}r(\cJ_{\tilde\sigma}(rt)+\sigma(r)t^\al).
\end{align*}

\medskip\noindent \emph{Step 2.4.} We show that
$$
\tilde\omega_{A_r}(t)\le C_{\bff,w}(\tilde\sigma(rt)+t^\al),\qquad \cJ_{\tilde\omega_{A_r}}(t)\le C_{\bff,w}(\cJ_{\tilde\sigma}(rt)+t^\al).
$$
Indeed, for given $t\in(0,1)$, we take $N\in \mathbb{N}$ such that $\ka^{-N}t\le 1<\ka^{-(N+1)}t$. Then 
$$
\tilde\omega_{A_r}(t)=\sum_{i=1}^N\ka^{i\al}\omega_{A_r}(\ka^{-i}t)+\sum_{i=N+1}^\infty\ka^{i\al}\omega_{A_r}(1).
$$
As we have seen in \emph{Step 2.3}, $\ka^{N+1}<t$ implies $\sum_{i=N+1}^\infty\ka^{i\al}\omega_{A_r}(1)\le C\omega_{A_r}(1)t^\al$. By using this and $\omega_{A_r}(t)\le C_{\bff,w}(\sigma(rt)+t)$, we have 
\begin{align*}
    \tilde\omega_{A_r}(t)&\le C_{\bff,w}\sum_{i=1}^N\ka^{i\al}\sigma(\ka^{-i}rt)+C_{\bff,w}t\sum_{i=1}^N\ka^{i(\al-1)}+C_{\bff,w}(\sigma(r)+1)t^\al\\
    &\le C_{\bff,w}\sum_{i=1}^\infty\ka^{i\al}\sigma(\ka^{-i}rt)[\ka^{-i}rt\le1]+C_{\bff,w}t\ka^{N(\al-1)}+C_{\bff,w}t^\al\\
    &\le C_{\bff,w}\tilde\sigma(rt)+C_{\bff,w}t^\al,
\end{align*}
where in the last inequality we used $t<\ka^N$ (which gives $\ka^{N(\al-1)}<t^{\al-1}$). As we have seen in \emph{Step 2.3}, this estimate for $\tilde\omega_{A_r}$ implies $\cJ_{\tilde\omega_{A_r}}(t)\le C_{\bff,w}(\cJ_{\tilde\sigma}(rt)+t^\al)$.

\medskip\noindent\emph{Step 3.} In this step, we simplify the estimates \eqref{eq:grad-w-p-infty} and \eqref{eq:grad-w-mod} by making use of the results in \emph{Step 2}. We start with \eqref{eq:grad-w-p-infty}.
\begin{align}
    \label{eq:grad-w-p-infty-1}
    \|\D(w-P)\|_{L^\infty(B_{r/2}(\frac{11}{10}r\vec e_n))}\le C_{\bff,w}(\cJ_\sigma(r)+\cJ_{\tilde\sigma}(r)+\sigma(r))\le C_{\bff,w}\cJ_{\sigma+\tilde\sigma}(r).
\end{align}
Regarding \eqref{eq:grad-w-mod}, we have for $x^1,\,x^2\in B_{r/2}(\frac{11}{10}r\vec e_n)$
\begin{multline*}
    |\D w(x^1)-\D w(x^2)|\\
    \le C_{\bff,w}\cJ_\sigma(r)\left(\frac{|x^1-x^2|}r\right)^\al+C_{\bff,w}\left(\cJ_{\tilde\sigma}(|x^1-x^2|)+\sigma(r)\left(\frac{|x^1-x^2|}r\right)^\al\right)\\
    +C_{\bff,w}(\cJ_\sigma(r)+\cJ_{\tilde\sigma}(r))\left(\cJ_{\tilde\sigma}(|x^1-x^2|)+\left(\frac{|x^1-x^2|}r\right)^\al\right).
\end{multline*}
Since $|x^1-x^2|<r$, we have $\sigma(r)\left(\frac{|x^1-x^2|}r\right)^\al\le \sigma(|x^1-x^2|)$ by $\al$-nonincreasing condition $(D3)$. Similarly, we also have $\cJ_\sigma(r)\left(\frac{|x^1-x^2|}r\right)^\al\le\cJ_\sigma(|x^1-x^2|)$ and $\cJ_{\tilde\sigma}(r)\left(\frac{|x^1-x^2|}r\right)^\al\le\cJ_{\tilde\sigma}(|x^1-x^2|)$. Therefore,
\begin{align*}
    |\D w(x^1)-\D w(x^2)|\le C_{\bff,w}\cJ_{\sigma+\tilde\sigma}(|x^1-x^2|).
\end{align*}
As the above argument in this section is valid when we vary the center of the ball $B_{r/2}\left(\frac{11}{10}r\vec e_n\right)$, we see that for any $x^0\in\overline{B^+_{1/2}}$ and $0<r<r_0/3$
\begin{align}
    \label{eq:grad-w-semi-norm}
    [\D w]_{C^{0,\cJ_{\sigma+\tilde\sigma}}\left(B_{r/2}\left(x^0+\frac{11}{10}r\vec e_n\right)\right)}\le C_{\bff,w}.
\end{align}

\medskip\noindent\emph{Step 4.} We are now ready to complete the proof.

Since $0<r<r_0/3$ is arbitrary, \eqref{eq:grad-w-p-infty-1} gives $|\D(w-P)(0)|=0$, which combined with Proposition~\ref{prop:Dini-poly-approx} yields $$
|\D w(0)|\le C_{\bff,w}.
$$ 
In addition, the identity $\D w(0)=\D P(0)$, together with \eqref{BCP}, implies 
$$
\mean{A(0)\D w(0)+\bff(0),\vec e_n}=0.
$$
Again, by varying the center of balls and half-balls $x_0$ and repeating the argument in this section, we can obtain \eqref{eq:w-BC}.

Next, we claim that for $\rho_0>0$ small universal
$$
[\D w]_{C^{0,\cH_{\sigma+\tilde\sigma}}(B_{\rho_0}^+)}\le C_{\bff,w},
$$
which together with $|\D w(0)|\le C_{\bff,w}$ will conclude the proof. For this purpose, let $x,\,y\in B_{\rho_0}^+$ be given. We write $x=(x',x_n)$ and $y=(y',y_n)$. Let $m$ be the slope of the line connecting $x$ and $y$, i.e., $m=\frac{|x_n-y_n|}{|x'-y'|}$. We let $m=\infty$ when $x'=y'$, and divide our proof into two cases
$$
\text{either} \quad m\ge1\, (\text{including } m=\infty)\quad \text{or}\quad 0\le m<1.
$$
\medskip\noindent\emph{Case 1.} We first consider the case $m\ge1$. Without loss of generality, we may assume $x_n\ge y_n$. Let $z:=(y',x_n)$. We will estimate separately the horizontal variation $|\D w(x)-\D w(z)|$ and the vertical variation $|\D w(z)-\D w(y)|$, which will give the bound for $|\D w(x)-\D w(y)|$ by triangle inequality. 

We first deal with the horizontal variation $|\D w(x)-\D w(z)|$. If $x'=y'$, then $x=z$, thus $|\D w(x)-\D w(z)|=0$. Hence, we may assume $x'\neq y'$, and let $e:=\frac{(y'-x',0)}{|y'-x'|}$ and $r_1:=\frac{10}{11}|x-z|=\frac{10}{11}|x'-y'|$. Note that $z=x+|x'-y'|e=x+\frac{11}{10}r_1e$. From $m\ge1$ we infer $\frac{11}{10}r_1=|x'-y'|\le x_n-y_n\le x_n$. This allows us to apply \eqref{eq:grad-w-semi-norm} to obtain
\begin{multline*}
    [\D w]_{C^{0,\cJ_{\sigma+\tilde\sigma}}(B_{r_1/2}(x))}+[\D w]_{C^{0,\cJ_{\sigma+\tilde\sigma}}(B_{r_1/2}(x+r_1e/2))}+[\D w]_{C^{0,\cJ_{\sigma+\tilde\sigma}}(B_{r_1/2}(x+r_1e))}\\
    \le C_{\bff,w}.
\end{multline*}
This, together with $z=x+\frac{11}{10}r_1e\in B_{h_1/2}(x+r_1e)$, yields
\begin{align*}
    |\D w(x)-\D w(z)|& \le|\D w(x)-\D w(x+r_1e/2)|+|\D w(x+r_1e/2)-\D w(x+r_1e)|\\
    &\qquad+|\D w(x+r_1e)-\D w(z)|\\
    &\le C_{\bff,w}\cJ_{\sigma+\tilde\sigma}(r_1/2)\le C_{\bff,w}\cJ_{\sigma+\tilde\sigma}(|x-y|).
\end{align*}

Now we estimate the vertical variation $|\D w(z)-\D w(y)|$. We take $r_2>0$ such that $\frac{11}{10}r_2\in (x_n-y_n,x_n)$. We consider sequences $\rho_k$ and $s_k$, $k\ge0$, defined by
$$
\rho_k:=(r_2/2)\left(\frac6{11}\right)^k,\quad \begin{cases}
    s_0:=x_n,\\
    s_k:=x_n-\sum_{i=0}^{k-1}\rho_i,\,\,\,k\ge1.
\end{cases}
$$
Here, $\rho_k$ and $s_k$ satisfy the following properties:
\begin{align*}
    \begin{cases}
        \text{- }(y',s_{k+1})\subset\partial B_{\rho_k}(y',s_k),\quad k\ge0,\\
        \text{- }[y,z]\subset \cup_{k=0}^\infty B_{\rho_k}(y',s_k),\text{ i.e., the chain of balls $\{B_{\rho_k}(y',s_k)\}$ connects $z$ and $y$,}\\
        \text{- }\rho_k\le\frac5{11}s_k \text{ (which enables us to apply \eqref{eq:grad-w-semi-norm} for $B_{\rho_k}(y',s_k)$)}.
    \end{cases}
\end{align*}
Indeed, the first property is clear by the definition of $s_k$. The second follows from the simple computation
$$
\lim_{k\to\infty}s_k=x_n-\sum_{i=0}^\infty \rho_i=x_n-\frac{11}{10}r_2<y_n.
$$
For the last one, we observe that for $k\ge 1$ (when $k=0$, it is trivial as $x_n>\frac{11}{10}r_2$)
\begin{align*}
    s_k&=x_n-\sum_{i=0}^{k-1}\rho_i=x_n-\frac{11}{10}r_2\left(1-\left(\frac6{11}\right)^k\right)\ge \frac{11}{10}r_2-\frac{11}{10}r_2\left(1-\left(\frac6{11}\right)^k\right)\\
    &=\frac{11}{10}r_2\left(\frac6{11}\right)^k=\frac5{11}\rho_k.
\end{align*}

We then have 
\begin{align*}
    |\D w(z)-\D w(y)|&\le \sum_{k=0}^\infty|\D w(y',s_k)-\D w(y',s_{k+1})|\le C_{\bff,w}\sum_{k=0}^\infty\cJ_{\sigma+\tilde\sigma}(\rho_k)\\
    &\le C_{\bff,w}\int_0^{r_2/2}\frac{\cJ_{\sigma+\tilde\sigma}(t)}t\,dt=C_{\bff,w}\cH_{\sigma+\tilde\sigma}\left(r_2/2\right).
\end{align*}
As the constant $C_{\bff,w}$ is independent of $r_2$ and $r_2\in\left(\frac{10}{11}(x_n-y_n),\frac{10}{11}x_n\right)$ can be chosen arbitrarily, we can pick a particular $r_2$ satisfying $r_2<2(x_n-y_n)$ to get 
$$
|\D w(z)-\D w(y)|\le C_{\bff,w}\cH_{\sigma+\tilde\sigma}(x_n-y_n)\le C_{\bff,w}\cH_{\sigma+\tilde\sigma}(|x-y|).
$$
Therefore, 
$$
|\D w(x)-\D w(y)|\le |\D w(x)-\D w(z)|+|\D w(z)-\D w(y)|\le C_{\bff,w}\cH_{\sigma+\tilde\sigma}(|x-y|).
$$

\medskip\noindent\emph{Case 2.} Now we deal with the case when the slope $m$ is between $0$ and $1$. Note that $x'-y'\neq0$ unless $x=y$.
For $e=\frac{(y'-x',0)}{|y'-x'|}$ as in \emph{Case 1}, let $e':=\frac{y'-x'}{|y'-x'|}$. To use the result in \emph{Case 1}, we consider the following lines of slope $1$
$$
\mathcal L_x:=\{x+(re',r)\,:\,r>0\},\quad \mathcal L_y:=\{y+(-re',r)\,:\,r>0\}.
$$
The intersection $\mathcal{L}_{x}\cap \mathcal L_y$ cannot be empty since the slope $m$ of the line connecting $x$ and $y$ is less than $1$. Thus $\mathcal L_x\cap \mathcal L_y=\{z\}$ for some point $z=(z',z_n)$. As the slope of both $\mathcal L_x$ and $\mathcal L_y$ is $1$ and $z'\in[x',y']$, we have
$$
|x-z|=\sqrt2|x'-z'|\le\sqrt2|x'-y'|\le\sqrt2|x-y|,
$$
and similarly 
$$
|y-z|\le\sqrt2|x-y|.
$$
Therefore, by using those inequalities and the result in \emph{Case 1}, we obtain
\begin{align*}
    |\D w(x)-\D w(y)|&\le |\D w(x)-\D w(z)|+|\D w(z)-\D w(y)|\\
    &\le C_{\bff,w}\cH_{\sigma+\tilde\sigma}(|x-z|)+C_{\bff,w}\cH_{\sigma+\tilde\sigma}(|y-z|)\\
    &\le C_{\bff,w}\cH_{\sigma+\tilde\sigma}(\sqrt2|x-y|)\\
    &\le C_{\bff,w}\cH_{\sigma+\tilde\sigma}(|x-y|).
\end{align*}
Here, in the last inequality we used that $\al$-nonincreasing condition $(D3)$ is satisfied for $\cH_{\sigma+\tilde\sigma}$.
\end{proof}

\subsection{The $C^k$ estimate}

Aim of this section is the proof of Theorem \ref{teok}. This is done by an induction argument in \cite{TerTorVit22}. First, we adapt some lemmas in \cite{TerTorVit22} to general moduli of continuity.

\begin{lemma}\label{lem1}
Let $k\in\mathbb N$ and $\omega$ be a modulus of continuity. Let $f\in C^{k+1,\omega}(B_1)$ with $f(x',0)=0$. Then $f/x_n\in C^{k,\omega}(B_1)$.
\end{lemma}
\proof
By direct computation
$$f(x',x_n)=\int_0^{x_n}\partial_nf(x',t)dt=x_n\int_0^1\partial_nf(x',sx_n)ds.$$
Hence
\begin{equation*}
\frac{f(x',x_n)}{x_n}=\int_0^1\partial_nf(x',sx_n)ds
\end{equation*}
and regularity follows by Leibniz rule. Indeed, if we consider a multiindex $|\beta|=k$, then taking two points $x,\overline x$
\begin{align*}
|D^\beta(f/x_n)(x)-D^\beta(f/x_n)(\overline x)|&\leq \int_0^1|D^\beta\partial_nf(x',sx_n)-D^\beta\partial_nf(\overline x',s\overline x_n)|ds\\
&\leq  C\int_0^1 \omega(|(x',sx_n)-(\overline x',s\overline x_n)|) ds\leq C\omega(|x-\overline x|).
\end{align*}
\endproof

\begin{lemma}\label{lem2}
Let $a>-1$, $k\in\mathbb N$ and $\omega$ be a modulus of continuity satisfying the $\alpha$-nonincreasing condition for some $\alpha\in(0,1]$. Let $g\in C^{k,\omega}(B_1)$. Then
\begin{equation*}
\psi(x',x_n)=\frac{1}{x_n|x_n|^a}\int_0^{x_n}|t|^ag(x',t)dt
\end{equation*}
belongs to $C^{k,\omega}(B_1)$.
\end{lemma}
\proof
The proof is done by induction arguing as in \cite[Lemma 2.4]{TerTorVit22}. We just give some details of the starting step $k=0$, adapting the proof of \cite[Lemma 7.5]{SirTerVit21a} to the case of general moduli. We can rewrite $\psi$ as
\begin{equation*}
\psi(x',x_n)=\frac{1}{x_n|x_n|^a}\int_0^{x_n}|t|^a(g(x',t)-g(x',0))dt +\frac{g(x',0)}{1+a}.
\end{equation*}
Let us do some preliminary remarks. First, we can prove the result on the upper half ball; that is, considering $x,\overline x\in \overline{B_1^+}$. Indeed the situation in the lower half ball is the same, and then the full uniform continuity follows by 
$$\lim_{x_n\to0+}\psi(x'_0,x_n)=\lim_{x_n\to0-}\psi(x'_0,x_n)=\frac{g(x'_0,0)}{a+1}.$$
Then, we remark that the result in the upper half ball follows by triangle inequality if we just prove that $|\psi(x',x_n)-\psi(x',\overline x_n)|\leq C\omega(|x_n-\overline x_n|)$ being the control over the tangential variation trivial. So, fixed $x'\in B_1'$, we consider the set $\mathcal A$ of $0\leq y\leq1$ such that $(x',y)\in\overline B_1^+$. Then we define $\mathcal A_1:=\{(y_1,y_2)\in\mathcal A\times\mathcal A \ : \ y_1< y_2/2\}$, $\mathcal A_2:=\{(y_1,y_2)\in\mathcal A\times\mathcal A \ : \ y_2/2\le y_1< y_2\}$. In $\mathcal A_1$, by monotonicity of $\omega$, we have $\omega(y_1/2)\leq\omega(y_2/2)\leq \omega(y_2-y_1)$, and hence
\begin{eqnarray*}
\frac{|\psi(x',y_1)-\psi(x',y_2)|}{\omega(y_2-y_1)}&\leq& \sum_{i=1}^2\frac{1}{y_i^{1+a}\omega(y_i/2)}\int_0^{y_i}t^a|g(x',t)-g(x',0)|dt\\
&\leq& C \sum_{i=1}^2\frac{\omega(y_i)}{\omega(y_i/2)}\leq C,
\end{eqnarray*}
where the last inequality follows by the $\alpha$-nonincreasing condition; that is, $\omega(r)r^{-\alpha}$ is nonincreasing for some $\alpha\in(0,1]$. Then, in $\mathcal A_2$
\begin{eqnarray*}
|\psi(x',y_1)-\psi(x',y_2)|&\leq& \frac{1}{y_2^{1+a}}\int_{y_1}^{y_2}t^a|g(x',t)-g(x',0)|dt\\ 
&&+ \left(\frac{1}{y_1^{1+a}}-\frac{1}{y_2^{1+a}}\right)\int_{0}^{y_1}t^a|g(x',t)-g(x',0)|dt\\
&=& I+II.
\end{eqnarray*}
For the first term $I$, using that $\omega(y_2/2)(y_2/2)^{-1}\leq\omega(y_2-y_1)(y_2-y_1)^{-1}$, let us remark that there exists $y_1<\xi<y_2$ such that
\begin{eqnarray*}
\frac{I}{\omega(y_2-y_1)}&=& \frac{y_2-y_1}{y_2^{1+a}\omega(y_2-y_1)}\xi^a|g(x',\xi)-g(x',0)|\\ 
&\leq& C\frac{\omega(\xi)}{\omega(y_2/2)}\left(\frac{\xi}{y_2}\right)^a\leq C\frac{\omega(\xi)}{\omega(y_2)}\left(\frac{\xi}{y_2}\right)^a\le C,
\end{eqnarray*}
where in the last inequality we used the fact that $1/2<\xi/y_2<1$ and. In order to deal with the second term $II$, we remark that
$$ \left(\frac{1}{y_1^{1+a}}-\frac{1}{y_2^{1+a}}\right)= \left(1-\left(1-\frac{y_2-y_1}{y_2}\right)^{1+a}\right)\frac{1}{y_1^{1+a}}\leq \frac{c}{y_1^{1+a}}\frac{y_2-y_1}{y_2}.$$
Hence
\begin{equation*}\frac{II}{\omega(y_2-y_1)}\leq \frac{c}{1+a} \frac{\omega(y_1)}{y_2}\frac{y_2-y_1}{\omega(y_2-y_1)}\leq  \frac{c}{1+a} \frac{\omega(y_1)}{y_2}\frac{y_2}{\omega(y_2)} \le\frac{c}{1+a}.\qedhere\end{equation*}
%in order to get $\varphi\in C^{0,\alpha}(B_1^+)$ and $\varphi\in C^{0,\alpha}(B_1^-)$. Hence, it remains to prove that $\varphi$ is actually continuous across $\Sigma$, which is true since
%$$\lim_{y\to0+}\varphi(x_0,y)=\lim_{y\to0-}\varphi(x_0,y)=\frac{g(x_0,0)}{a+1}.$$
%Let us assume that the result is true for a generic integer $k\in\mathbb N$ and let us prove the result for $k+1$. In other words, we claim that $\partial_{x_i}\varphi\in C^{k,\alpha}(B_1)$, for any $i=1,...,n-1$, and $\partial_y\varphi\in C^{k,\alpha}(B_1)$. Here we are assuming $g\in C^{k+1,\alpha}(B_1)$, hence
%$$\partial_{x_i}\varphi(x,y)=\frac{1}{y|y|^a}\int_0^y|t|^a\partial_{x_i}g(x,t)dt$$
%which belongs to $C^{k,\alpha}(B_1)$ by inductive hypothesis. Then, let us express the function $\varphi$ as
%\begin{equation*}
%\varphi(x,y)=\frac{1}{y|y|^a}\int_0^y|t|^a(g(x,t)-g(x,0))dt+\frac{g(x,0)}{a+1}.
%\end{equation*}
%Hence,
%\begin{equation*}
%\partial_y\varphi(x,y)=-\frac{a+1}{y|y|^{a+1}}\int_0^y|t|^{a+1}\frac{g(x,t)-g(x,0)}{t}dt+\frac{g(x,y)-g(x,0)}{y}.
%\end{equation*}
%By Lemma \ref{lem1}, $(g(x,y)-g(x,0))/y\in C^{k,\alpha}(B_1)$ and hence we can conclude using again the inductive hypothesis.
%
\endproof

The previous lemma implies the following %\gio{scriverlo nello stesso Lemma di prima dato che segue da quello precedente dato che la $\varphi$ di Rem 2.3 è quella di Lemma 2.2. moltiplicata con $y$}
\begin{remark}\label{lem3}
Under the hypothesis of Lemma \ref{lem2}, we have
\begin{equation*}
\varphi(x',x_n)=\frac{1}{|x_n|^a}\int_0^{x_n}|t|^ag(x',t)dt
\end{equation*}
possesses partial derivative $\partial_n\varphi\in C^{k,\omega}(B_1)$.
In fact
$$\partial_n\varphi(x',x_n)=-\frac{a}{|x_n|^ax_n}\int_0^{x_n}|t|^ag(x',t)dt+g(x',x_n).$$
\end{remark}

\begin{proof}[Proof of Theorem \ref{teok}]
The proof follows the one of \cite[Lemma 2.7]{TerTorVit22}, but we repeat it here for the sake of completeness.

We procede by induction. The result in case $k=0$ is true by Theorem \ref{thm:Dini-ratio-reg} and by a standard covering argument, which gives also the boundary condition \eqref{Neumann}. Let us assume the result true for $k\in\mathbb N$ and prove it for $k+1$. Thus, we are assuming $A,\bff\in C^{k+1,\omega}(B_1^+)$ where $\omega$ is a $j$-Dini function, and we would like to prove that $w_i:=\partial_{i}w\in C^{k+1,\sigma}(B_r^+)$ for any $i=1,...,n$ and for some $(j-2)$-Dini function $\sigma$.
Notice that the tangential derivatives $w_i$ with $i=1,...,n-1$ do solve
\begin{equation*}\label{evenLa3x}
-\div\left(x_n^aA\nabla w_i\right)=\div\left(x_n^a(\partial_{i}\bff+\partial_{i}A\nabla w)\right) \qquad\mathrm{in \ } B_1^+,
\end{equation*}
in the sense that
\begin{equation*}
-\int_{B_1^+}x_n^a\mean{A\nabla w_i,\nabla\phi}=\int_{B_1^+}x_n^a\mean{\partial_{i}\bff+\partial_{i}A\nabla w,\nabla\phi}
\end{equation*}
for every test function $\phi\in C^{\infty}_0(B_1)$ (when $a\geq1$ test functions can be taken in $C^{\infty}_0(B_1\setminus\Sigma)$).

Let us remark here that the assumption $A,\bff\in C^{k+1,\omega}$ implies $A,\bff\in C^{k,1}$ which gives, by standard Schauder estimates for degenerate equations \cite[Theorem 1.1]{TerTorVit22}, $w\in C^{k+1,\alpha}$ for any $\alpha\in(0,1)$.
Hence, the field $\partial_{i}\bff+\partial_{i}A\nabla w$ belongs to $C^{k,\omega}(B_r^+)$, by inductive hypothesis we have
\begin{equation}\label{xireg}
w_i\in C^{k+1,\sigma}(B_r^+)\qquad \mathrm{for \  any \ } i=1,...,n-1.
\end{equation}

Now, equation \eqref{evenLa3} can be rewritten as
\begin{equation*}
-\div\left(A\nabla w\right)=\frac{a\mean{A\nabla w+\bff,\vec e_{n}}}{x_n}+\div \bff.
\end{equation*}
Hence
\begin{equation}\label{eqphi}
\partial_n\varphi+\frac{a}{x_n}\varphi=g:=-\div \bff+\partial_n\bff_{n}-\sum_{i=1}^{n-1}\partial_{i}(\mean{A\nabla w,\vec e_i}),
\end{equation}
where $g\in C^{k,\omega}(B_r^+)$ and $\varphi:=\mean{A\nabla w+\bff,\vec e_{n}}$ with $\varphi(x',0)=0$ by \eqref{Neumann}. Since $\partial_n\varphi+\frac{a}{x_n}\varphi=x_n^{-a}\partial_n\left(x_n^a\varphi\right)$, then one would like to prove that
\begin{equation}\label{phireg}
\varphi(x',x_n)=\frac{1}{x_n^a}\int_0^{x_n}t^ag(x',t)dt\in C^{k+1,\sigma}(B_r^+).
\end{equation}
Eventually, this last information together with \eqref{xireg} would give $w\in C^{k+2,\sigma}(B_r^+)$. In fact, isolating the term $w_{nn}:=\partial^2_{nn}w$ in the left hand side of \eqref{eqphi}, one gets the desired regularity also for this last derivative from the equation once one observes that the uniform ellipticity of $A$ gives the uniform bound from below
\begin{equation}\label{ellipticn}
a_{nn}=\mean{A\vec e_{n},\vec e_{n}}\geq 1/\Lambda>0.
\end{equation}
In order to prove \eqref{phireg} it is enough to remark that \eqref{xireg} and the definition of $\varphi$ immediately give $\partial_{i}\varphi\in C^{k,\sigma}(B_r^+)$ for $i=1,...,n-1$. Then, by Remark \ref{lem3} one also gets $\partial_n\varphi\in C^{k,\sigma}(B_r^+)$.
\end{proof}

\section{Higher order boundary Harnack principles in Dini type domains}\label{sec:4}

This section is devoted to the proof of Theorem \ref{thm:Dini-bdry-harnack} and Corollary \ref{schauderR(u)} as consequences of Theorem \ref{teok}.

\subsection{One side higher order boundary Harnack principle}

Let us consider two functions $u,v$ solving \eqref{BHconditions}; that is,
\begin{equation*}
\begin{cases}
-Lv=f &\mathrm{in \ }\Omega\cap B_1,\\
-Lu=g &\mathrm{in \ }\Omega\cap B_1,\\
u>0 &\mathrm{in \ }\Omega\cap B_1,\\
u=v=0, \quad \partial_{\nu} u<0&\mathrm{on \ }\partial\Omega\cap B_1.
\end{cases}
\end{equation*}
We recall that $\Omega\subset\R^n$ is a bounded domain, $n\geq2$, the operator $L$ is defined in \eqref{eq:div-operator}, the variable coefficient matrix $A=(a_{ij})_{i,j=1}^n$ is symmetric and satisfies \eqref{eq:assump-coeffi}, and $\nu$ is the outward unit normal vector to $\Omega$ on $\partial\Omega$.

Let $k\in\mathbb N$, $j\in\mathbb N\setminus\{0,1,2\}$ and $\omega$ be a $j$-Dini function as in Definition \ref{dinidef}. Let us assume that $A,f,g\in C^{k,\omega}(\Omega\cap B_1)$ and $\partial\Omega\in C^{k+1,j-Dini}$ in the sense of Definition \ref{diniboundary}.
%In other words, there exists $r>0$ and points $x^i\in\partial\Omega$, such that after a possible rotation of coordinate axes, one has
%$$
%    B_r(x^i)\cap\Omega=B_r(x^i)\cap\{x_n>\g_i(x_1,\cdots,x_{n-1})\}
%    $$
%    for a $C^{k,\sigma}$ function $\g_i:\R^{n-1}\to\R$ for some $j$-Dini function $\sigma$.
Since our estimates are of local type, after possible rotations and dilations, we may restrict without loss of generality to parametrize $$B_1\cap\Omega=B_1\cap\{x_n>\g(x')\},$$ with $x'=(x_1,\cdots,x_{n-1})$, $\gamma(0)=0$, $\nabla_{x'}\gamma(0)=0$, $\gamma\in C^{k+1,\omega}(B_1')$, $[D^{k+1}\gamma]_{C^{0,\omega}}\leq1$ and $\omega$ is the $j$-Dini modulus of $A,f,g$.

\begin{remark}[Around the Hopf-Oleinik boundary point principle]
Let us remark here that, under our assumptions on the boundary when $k=0$, the \emph{interior $C^{1,\mathcal D}$-paraboloid condition} in \cite{ApuNaz19} is satisfied; that is,
$$\gamma(x')\leq |x'|\,\omega(|x'|)\qquad\mathrm{in \ }B_1'.$$
Then, when $g\geq0$ the condition $\partial_\nu u<0$ on $\partial\Omega$ is implied by \cite[Theorem 2.1]{ApuNaz19}.
\end{remark}

\begin{proof}[Proof of Theorem \ref{thm:Dini-bdry-harnack}]
Let us consider the local diffeomorphism which straighten the boundary $\partial\Omega$; that is,
\begin{equation}\label{diffeo}
\Phi(x',x_n)=(x',x_n+\g(x')),
%=(\overline x',\overline x_n),
\end{equation}
which is of class $C^{k+1,\omega}$. There exists a small enough $R>0$ such that
$\Phi(B_R\cap\{x_n>0\})\subseteq B_1\cap\{x_n>\g(x')\}$; that is, it is a subset of the domain where the original equation is satisfied and $\Phi(0)=\Phi^{-1}(0)=0$. Additionally, the boundary $B_R\cap\{x_n=0\}$ is mapped into the boundary $B_1\cap\{x_n=\g(x')\}$.
The Jacobian associated with $\Phi$ is given by
\begin{align*}
J_\Phi(x)=J_\Phi(x')=\left(\begin{array}{c|c}
\mathbb{I}_{n-1}&{\mathbf 0}\\\hline
(\nabla_{x'} \g(x'))^T&1
\end{array}\right), \qquad \mathrm{with}\quad |\mathrm{det} \, J_\Phi(x')|\equiv1.
\end{align*}
We remark that $\Phi$ is bi-Lipschitz continuous. Moreover, the new outward unit normal vector is
\begin{equation}\label{normaltransf}
-\vec e_n=(J_{\Phi})^T(\nu\circ\Phi)\sqrt{1+|\nabla_{x'}\g|^2}.
\end{equation}
This can be checked having that 
\begin{equation*}
\nu\circ\Phi(x',0)=\frac{(\nabla_{x'}\g(x'),-1)^T}{\sqrt{1+|\nabla_{x'}\g(x')|^2}}.
\end{equation*}
Then, up to dilations, $\tilde v=v\circ\Phi$ and $\tilde u=u\circ\Phi$ solve
\begin{equation*}
\begin{cases}
-\div\left(\tilde A\nabla \tilde v\right)=\tilde f &\mathrm{in \ }B_1^+,\\
-\div\left(\tilde A\nabla \tilde u\right)=\tilde g &\mathrm{in \ }B_1^+,\\
\tilde u>0 &\mathrm{in \ }B_1^+,\\
\tilde u=\tilde v=0, \qquad -\partial_n \tilde u<0&\mathrm{on \ }B'_1,
\end{cases}
\end{equation*}
where $\tilde f=f\circ\Phi$, $\tilde g=g\circ\Phi$ and
\begin{equation*}
\tilde A=(J_\Phi^{-1}) (A \circ\Phi) (J_\Phi^{-1})^T.
\end{equation*}

As we have already remarked, we are going to prove the desired regularity for $\tilde w=\tilde v/\tilde u$ which will follow from Theorem \ref{teok}. Then, the regularity is inherited by $w$ through composition back with the diffeomorphism $w=\tilde w\circ\Phi^{-1}$.

We remark that, after the diffeomorphism, $\tilde f,\tilde g, \tilde A$ belong to $C^{k,\omega}(B_1^+)$, and hence Corollary \ref{cor:reg:k} implies $\tilde u,\tilde v\in C^{k+1,\sigma}(B_r^+)$ for any $r<1$ where $\sigma$ is $(j-1)$-Dini if $\omega$ is $j$-Dini (in the present case we are assuming $j\ge3$ and hence $j-1\ge2$). Let us remark that $D^{\beta}u$ and $D^{\beta}v$, for any multiindex with $|\beta|=k+1$, have the same modulus of continuity $\sigma:=\sigma_u=\sigma_v$. The fact that $\tilde w=w\circ\Phi=\tilde v/\tilde u\in C^{k,\sigma}$, follows once we notice that
$$\frac{\tilde v}{\tilde u}=\frac{\tilde v}{x_n}\left(\frac{\tilde u}{x_n}\right)^{-1}.$$
Indeed, by combining that $\tilde u>0$ in $B_1^+$ with $-\partial_n \tilde u<0$ on $B'_1$, we deduce that
\begin{equation}\label{mutilde}
\frac{\tilde u}{x_n}\geq\mu>0 \quad\mathrm{up \  to}\quad \Sigma=\{x_n=0\}.
\end{equation}
Finally, the result follows since, by Lemma \ref{lem1}, the ratio is the product of two $C^{k,\sigma}$ functions. In order to prove that the ratio $\tilde w$ actually belongs to $C^{k+1}$, we use the fact that for any $0<r<1$, $\tilde w$ is solution to
\begin{equation}\label{eqdessav}
-\div\left(x_n^2\overline A\nabla \tilde w\right)=x_n^2\frac{\overline f}{x_n}\qquad\mathrm{in \ }B_r^+,
\end{equation}
where the new matrix and forcing term are
$$
\overline A=(\tilde u/x_n)^2\tilde A,\qquad \overline f=(\tilde u/x_n)\tilde f-(\tilde v/x_n)\tilde g=(\tilde u/x_n)\left(\tilde f-\tilde g\tilde w\right)
$$
which belong to $C^{k,\sigma}(B_r^+)$. Nevertheless coefficients of $\overline A$ do not vanish on $\Sigma$ and hence the matrix is still uniformly elliptic, thanks to the non degeneracy condition \eqref{mutilde}. In order to prove that $\tilde w$ is actually a weak solution of the weighted equation above, we just need to prove that it belongs to the weighted Sobolev space $H^{1,2}(B_r^+)=H^{1}(B_r^+,x_n^2dx)$. In fact, due to the strong degeneracy of the weight $x_n^2$ we already know that the weak formulation is done just testing the equation with any $\phi\in C^{\infty}_0(B_r^+)$; that is, smooth function with compact support away from $\Sigma$, and the equation is satisfied in a classic sense in compact subsets of $B_1^+$. Then, to prove that $\tilde w\in H^{1,2}(B_r^+)$ we reason as in \cite[Proposition 3.5]{TerTorVit22}. We would like to stress the fact that the Meyers-Serrin theorem $(H=W)$ is not known in this weighted setting \cite[Remark 2.3]{SirTerVit21a}; that is, a function with finite $H^{1,2}$-norm does not necessarily belong to the relevant Sobolev space. Hence, we need to prove that as $\delta\to0^+$ there exist functions $w_\delta\in C^{\infty}_0(\overline{B_r^+}\setminus\Sigma)$ converging to $\tilde w$ in $H^{1,2}$-norm. This is done in the following way: given a radially decreasing cut off function $\eta\in C^\infty_0(B_1)$ with $\eta\equiv 1$ in $B_r$ and $0\leq\eta\leq 1$, let us define $\varphi_\delta(x)=\eta(x) f_\delta(x_n)$ where $f_\delta$ is defined as
\begin{equation}\label{fdelta}
f_\delta(t)=\begin{cases}
0 &\mathrm{if} \  |t|\leq\delta^2\\
\log(|t|/\delta^2)/\log(1/\delta) &\mathrm{if} \ \delta^2< |t|<\delta\\
1 &\mathrm{if} \ |t|\geq\delta.
\end{cases}
\end{equation}
One can prove that $w_\delta=\varphi_\delta \tilde w$ are uniformly bounded in $H^{1,2}(B_1^+)$ with respect to $\delta$, and they converge to $\tilde w$ in $B_r^+$. Since the approximations have compact support away from $\Sigma$, it is enough to prove the finiteness of the weighted norm in order to deduce a uniform bound in $\delta$ in the Sobolev norm and this is done by applying the Hardy inequality contained in \cite[Lemma B.1]{SirTerVit21b}; that is,
$$
\frac{1}{4}\int_{B_1^+}\frac{\varphi^2}{x_n^2}\leq\int_{B_1^+}|\nabla\varphi|^2\qquad\forall \varphi\in H^1_0(B_1^+\cup\partial^+B_1^+),
$$
to $\varphi=\tilde v$.
Then, one can regularize the function $w_\delta$ by convolution with a standard family of mollifiers in order to get a sequence which is also smooth.

In order to apply our regularity theory for degenerate equations, notice that \eqref{eqdessav} can be rewritten as
\begin{equation*}
-\div\left(x_n^2\overline A\nabla \tilde w\right)=\div(x_n^2\bg),
\end{equation*}
with
$$\bg(x',x_n)=\frac{\vec e_n}{x_n^2}\int_0^{x_n}t\overline f(x',t)dt$$
which belongs to $C^{k,\sigma}$ thanks to Lemma \ref{lem2}, and having $\mean{\bg(x',0),\vec e_n}=\overline f(x',0)/2$.
Hence Theorem \ref{teok} implies the desired regularity; that is, $\tilde w\in C^{k+1,\overline\sigma}$ with $\overline\sigma$ a $(j-3)$-Dini function. Moreover, \eqref{Neumann} implies the boundary condition
\begin{equation*}
2\mean{\overline A\nabla\tilde w,\vec e_n}+\overline f=0\qquad\mathrm{on \ }\Sigma.
\end{equation*}
Eventually, by \eqref{normaltransf}, one obtains the validity of the boundary condition \eqref{boundaryHOBH}, see the proof of \cite[Theorem 2.10]{TerTorVit22} for further details.
\end{proof}

\subsection{Boundary Harnack principle on nodal domains}

Let us consider now two $L$-harmonic functions $u,v$ in $B_1$; that is, solving 
\begin{equation}\label{Lharmonic}
Lu=Lv=0\qquad\mathrm{in \ }B_1.
\end{equation}
We recall that $n\geq2$, the operator $L$ is defined in \eqref{eq:div-operator}, the variable coefficient matrix $A=(a_{ij})_{i,j=1}^n$ is symmetric and satisfies \eqref{eq:assump-coeffi} in $B_1$. Let $k\in\mathbb N$, $j\in\mathbb N\setminus\{0,1,2\}$ and $\omega$ be a $j$-Dini function as in Definition \ref{dinidef}. Let us assume that $A\in C^{k,\omega}(B_1)$.

Then we assume that the two solutions share their nodal sets; that is, $Z(u)\subseteq Z(v)$. Moreover, the nodal set $Z(u):=u^{-1}(\{0\})$ can be divided into a regular and a singular part
$$R(u)=\{x\in Z(u) \, : \, |\nabla u(x)|\neq0\},\qquad S(u)=\{x\in Z(u) \, : \, |\nabla u(x)|=0\}.$$
Here, we assume that $S(u)\cap B_1=\emptyset$, and since we are interested in local estimates across $R(u)$, we can localize the problem around a given regular point. Then the regularity across $R(u)$ is obtained by a standard covering argument. Hence, up to further dilations, rotations and translations we can suppose that
$$B_1=B_1\cap\left(\Omega^+\cup\Omega^-\cup\Gamma\right),$$
where the three sets above are respectively, the epigraph, the hypograph and the graph of a function $\gamma:B'_1\to\mathbb R$
$$\Omega^+=\{x_n>\gamma(x')\},\qquad \Omega^-=\{x_n<\gamma(x')\},\qquad \Gamma=\{x_n=\gamma(x')\}.$$
In other words, we are localizing at a regular point $0\in Z(u)$ around which the nodal set is regular and locally described as a graph of a regular function. By Dini's implicit function theorem, since $u$ belongs to $C^{k+1,\sigma}_\loc(B_1)$, then $\gamma\in C^{k+1,\sigma}_\loc(B'_1)$, and we can assume that $\gamma(0)=0$, $\nabla_{x'}\gamma(0)=0$ and $[D^{k+1}\gamma]_{C^{0,\sigma}}\leq1$. Moreover, by the interior counterpart of Corollary \ref{cor:reg:k}, $\sigma$ is $(j-1)$-Dini.

Let us define $\nu^+$ the outward unit normal vector to $\Omega^+$ on $\Gamma$ and $\nu^-=-\nu^+$ the outward unit normal vector to $\Omega^-$ on $\Gamma$.

\begin{proof}[Proof of Corollary \ref{schauderR(u)}]
The proof follows the same strategies developed in \cite[Section 3.3.2]{TerTorVit22}. First, the diffeomorphism $\Phi$ in \eqref{diffeo} allows to work on
$$B_1=B_1^+\cup B_1^-\cup B'_1,$$
after possible dilations. Let us remark that such a diffeomorphism is locally defined from both sides of the interface and is of class $C^{k+1,\sigma}$ in the full ball. Moreover, by defining 
$\tilde A=(J_\Phi^{-1}) (A \circ\Phi) (J_\Phi^{-1})^T\in C^{k,\sigma}(B_1)$,
one gets that $\tilde u=u\circ\Phi$ is solution to
$$\div(\tilde A\nabla\tilde v)=0\qquad\mathrm{in \ }B_1.$$
Although by Corollary \ref{cor:reg:k} solutions to the previous equation are in general $C^{k+1,\tilde\sigma}_\loc(B_1)$ with $\tilde\sigma$ a $(j-2)$-Dini function, by composition of regular functions one can ensure more; that is, $\tilde u\in C^{k+1,\sigma}(B_1)$. Hence, assuming without loss of generality that $\tilde u>0$ in $B_1^+$ and $\tilde u<0$ in $B_1^-$, the Hopf-Oleinik lemma in \cite[Theorem 2.1]{ApuNaz19} applied on both half balls give a nondegeneracy condition
$$\frac{1}{C}\leq \frac{\tilde u}{x_n} \leq C\qquad \mathrm{in \ }B_1,$$
with $\tilde u/x_n\in C^{k,\sigma}(B_1)$ by Lemma \ref{lem1}. Then one can apply the regularity result in Theorem \ref{teok} on $\tilde w=\tilde v/\tilde u$ on both half balls, getting for any $0<r<1$ that $\tilde w\in C^{k+1,\overline\sigma}(B_r^+)$ and $\tilde w\in C^{k+1,\overline\sigma}(B_r^-)$ with $\overline\sigma$ a $(j-3)$-Dini function. This is due to the fact that $\tilde w$ solves
\begin{equation*}
\div(x_n^2\overline A\nabla\tilde w)=0\qquad\mathrm{in \ }B_1,
\end{equation*}
with $\overline A=(\tilde u/x_n)^2\tilde A\in C^{k,\sigma}$, uniformly elliptic and symmetric. Hence, it remains to glue together the estimates from both sides. This is done by applying \cite[Lemma 2.11]{TerTorVit22} with the relevant changes for the Dini context once one proves continuity of $\tilde w$ across $\Sigma$. The latter fact, follows by $C^1$ regularity of $\tilde u,\tilde v$ across $\Sigma$. In fact, given $x_0=(x'_0,0)$ in $B'_1$
\begin{eqnarray*}
\lim_{t\to0^+}\tilde w(x'_0,t)&=&\lim_{t\to0^+}\frac{\tilde v(x'_0,t)-\tilde v(x'_0,0)}{\tilde u(x'_0,t)-\tilde u(x'_0,0)}\\
&=&\lim_{t\to0^+}\frac{\tilde v(x'_0,t)-\tilde v(x'_0,0)}{t}\left(\lim_{t\to0^+}\frac{\tilde u(x'_0,t)-\tilde u(x'_0,0)}{t}\right)^{-1}\\
&=&\frac{\partial_n\tilde v(x'_0,0)}{\partial_n\tilde u(x'_0,0)},
\end{eqnarray*}
and the same result holds considering the limit from below $\lim_{t\to0^-}\tilde w(x'_0,t)$.
\end{proof}

\section*{Acknowledgement}
The authors would like to thank Jacopo Somaglia for fruitful conversations on Remark 1.2 in \cite{ApuNaz16, ApuNaz22} and Gabriele Cora for discussions on weighted Poincaré and Hardy inequalities. S.V. is research fellow of Istituto Nazionale di Alta Matematica INDAM group GNAMPA and supported by the MUR funding for Young Researchers - Seal of Excellence (ID: SOE\_0000194. Acronym: ADE. Project Title: Anomalous diffusion equations: regularity and geometric properties of solutions and free boundaries). This project was initiated while the authors stayed at Institut Mittag-Leffler during the program \emph{Geometric Aspects of Nonlinear Partial Differential Equations.}

\section*{Declaration}

Funding and/or Conflicts of interests: The authors declare that there are no financial or non-financial conflict of interests.

%%%%%%%%%%%%%%%%%%%%%%%%%%%%%%%%%%%%%%


\begin{bibdiv}
\begin{biblist}


\bib{AllSha19}{article}{
   author={Allen, Mark},
   author={Shahgholian, Henrik},
   title={A new boundary Harnack principle (equations with right hand side)},
   journal={Arch. Ration. Mech. Anal.},
   volume={234},
   date={2019},
   number={3},
   pages={1413--1444},
   issn={0003-9527},
   review={\MR{4011700}},
   doi={10.1007/s00205-019-01415-3},
}


\bib{Anc78}{article}{
   author={Ancona, Alano},
   title={Principe de Harnack \`a la fronti\`ere et th\'{e}or\`eme de Fatou pour un
   op\'{e}rateur elliptique dans un domaine lipschitzien},
   language={French, with English summary},
   journal={Ann. Inst. Fourier (Grenoble)},
   volume={28},
   date={1978},
   number={4},
   pages={169--213, x},
   issn={0373-0956},
   review={\MR{513885}},
}

  
\bib{ApuNaz16}{article}{
   author={Apushkinskaya, Darya E.},
   author={Nazarov, Alexander I.},
   title={A counterexample to the Hopf-Oleinik lemma (elliptic case)},
   journal={Anal. PDE},
   volume={9},
   date={2016},
   number={2},
   pages={439--458},
   issn={2157-5045},
   review={\MR{3513140}},
   doi={10.2140/apde.2016.9.439},
   }


\bib{ApuNaz19}{article}{
   author={Apushkinskaya, Darya E.},
   author={Nazarov, Alexander I.},
   title={On the boundary point principle for divergence-type equations},
   journal={Atti Accad. Naz. Lincei Rend. Lincei Mat. Appl.},
   volume={30},
   date={2019},
   number={4},
   pages={677--699},
   issn={1120-6330},
   review={\MR{4030346}},
   doi={10.4171/RLM/867},
   }   



\bib{ApuNaz22}{article}{
   author={Apushkinskaya, Darya E.},
   author={Nazarov, Alexander I.},
   title={The normal derivative lemma and surrounding issues},
   language={Russian, with Russian summary},
   journal={Uspekhi Mat. Nauk},
   volume={77},
   date={2022},
   number={2 (464)},
   pages={3--68},
   issn={0042-1316},
   translation={
      journal={Russian Math. Surveys},
      volume={77},
      date={2022},
      number={2},
      pages={189--249},
      issn={0036-0279},
   },
   review={\MR{4461367}},
   doi={10.4213/rm10049},
}






\bib{BanGar16}{article}{
   author={Banerjee, Agnid},
   author={Garofalo, Nicola},
   title={A parabolic analogue of the higher-order comparison theorem of De
   Silva and Savin},
   journal={J. Differential Equations},
   volume={260},
   date={2016},
   number={2},
   pages={1801--1829},
   issn={0022-0396},
   review={\MR{3419746}},
   doi={10.1016/j.jde.2015.09.044},
}







\bib{BanBasBur91}{article}{
   author={Ba\~{n}uelos, Rodrigo},
   author={Bass, Richard F.},
   author={Burdzy, Krzysztof},
   title={H\"{o}lder domains and the boundary Harnack principle},
   journal={Duke Math. J.},
   volume={64},
   date={1991},
   number={1},
   pages={195--200},
   issn={0012-7094},
   review={\MR{1131398}},
   doi={10.1215/S0012-7094-91-06408-2},
}




\bib{BasBur91}{article}{
   author={Bass, Richard F.},
   author={Burdzy, Krzysztof},
   title={A boundary Harnack principle in twisted H\"{o}lder domains},
   journal={Ann. of Math. (2)},
   volume={134},
   date={1991},
   number={2},
   pages={253--276},
   issn={0003-486X},
   review={\MR{1127476}},
   doi={10.2307/2944347},
}

\bib{BasBur94}{article}{
   author={Bass, Richard F.},
   author={Burdzy, Krzysztof},
   title={The boundary Harnack principle for nondivergence form elliptic
   operators},
   journal={J. London Math. Soc. (2)},
   volume={50},
   date={1994},
   number={1},
   pages={157--169},
   issn={0024-6107},
   review={\MR{1277760}},
   doi={10.1112/jlms/50.1.157},
}

\bib{CafFabMorSal81}{article}{
   author={Caffarelli, Luis},
   author={Fabes, Eugene B.},
   author={Mortola, Stefano},
   author={Salsa, Sandro},
   title={Boundary behavior of nonnegative solutions of elliptic operators
   in divergence form},
   journal={Indiana Univ. Math. J.},
   volume={30},
   date={1981},
   number={4},
   pages={621--640},
   issn={0022-2518},
   review={\MR{620271}},
   doi={10.1512/iumj.1981.30.30049},
}



\bib{CafSil07}{article}{
   author={Caffarelli, Luis},
   author={Silvestre, Luis},
   title={An extension problem related to the fractional Laplacian},
   journal={Comm. Partial Differential Equations},
   volume={32},
   date={2007},
   number={7-9},
   pages={1245--1260},
   issn={0360-5302},
   review={\MR{2354493}},
   doi={10.1080/03605300600987306},
}


%\bib{CorMusNaz23}{article}{
%   author={Cora, Gabriele},
%   author={Musina, Roberta},
%   author={Nazarov, Alexander I.},
%   title={},
%   pages={},
%   date={2023},
%   status={in preparation},
% }
%






\bib{Dal77}{article}{
   author={Dahlberg, Bj\"{o}rn E. J.},
   title={Estimates of harmonic measure},
   journal={Arch. Rational Mech. Anal.},
   volume={65},
   date={1977},
   number={3},
   pages={275--288},
   issn={0003-9527},
   review={\MR{466593}},
   doi={10.1007/BF00280445},
}





\bib{DeSSav15}{article}{
   author={De Silva, Daniela},
   author={Savin, Ovidiu},
   title={A note on higher regularity boundary {H}arnack inequality},
   journal={Discrete Contin. Dyn. Syst.},
   volume={35},
   date={2015},
   number={12},
   pages={6155--6163},
   issn={1078-0947},
   review={\MR{3393271}},
   doi={10.3934/dcds.2015.35.6155},
   }




\bib{DeSSav16}{article}{
   author={De Silva, Daniela},
   author={Savin, Ovidiu},
   title={Boundary Harnack estimates in slit domains and applications to
   thin free boundary problems},
   journal={Rev. Mat. Iberoam.},
   volume={32},
   date={2016},
   number={3},
   pages={891--912},
   issn={0213-2230},
   review={\MR{3556055}},
   doi={10.4171/RMI/902},
}

\bib{DeSSav20}{article}{
   author={De Silva, Daniela},
   author={Savin, Ovidiu},
   title={A short proof of boundary Harnack principle},
   journal={J. Differential Equations},
   volume={269},
   date={2020},
   number={3},
   pages={2419--2429},
   issn={0022-0396},
   review={\MR{4093736}},
   doi={10.1016/j.jde.2020.02.004},
}



\bib{DeSSav22a}{article}{
   author={De Silva, Daniela},
   author={Savin, Ovidiu},
   title={On the boundary Harnack principle in H\"{o}lder domains},
   journal={Math. Eng.},
   volume={4},
   date={2022},
   number={1},
   pages={Paper No. 004, 12},
   review={\MR{4245993}},
   doi={10.3934/mine.2022004},
}

\bib{DeSSav22b}{article}{
   author={De Silva, Daniela},
   author={Savin, Ovidiu},
   title={On the parabolic boundary Harnack principle},
   journal={Matematica},
   volume={1},
   date={2022},
   number={1},
   pages={1--18},
   review={\MR{4482724}},
   doi={10.1007/s44007-021-00001-y},
}





\bib{Don12}{article}{
   author={Dong, Hongjie},
   title={Gradient estimates for parabolic and elliptic systems from linear laminates},
   journal={Arch. Ration. Mech. Anal.},
   volume={205},
   date={2012},
   number={1},
   pages={119--149},
   issn={0003-9527},
   review={\MR{2927619}},
   doi={10.1007/s00205-012-0501-z},
   }


\bib{DonEscKim18}{article}{
   author={Dong, Hongjie},
   author={Escauriaza, Luis},
   author={Kim, Seick},
   title={On {$C^1$}, {$C^2$}, and weak type-{$(1,1)$} estimates for linear elliptic operators: part {II}},
   journal={Math. Ann.},
   volume={370},
   date={2018},
   number={1-2},
   pages={447--489},
   issn={0025-5831},
   review={\MR{3747493}},
   doi={10.1007/s00208-017-1603-6},
   }


\bib{DonKim17}{article}{
   author={Dong, Hongjie},
   author={Kim, Seick},
   title={On {$C^1$}, {$C^2$}, and weak type-{$(1,1)$} estimates for linear elliptic operators},
   journal={Comm. Partial Differential Equations},
   volume={42},
   date={2017},
   number={3},
   pages={417--435},
   issn={0360-5302},
   review={\MR{3620893}},
   doi={10.1080/03605302.2017.1278773},
   }
   
   \bib{DonLeeKim20}{article}{
   author={Dong, Hongjie},
   author={Lee, Jihoon},
   author={Kim, Seick},
   title={On conormal and oblique derivative problem for elliptic equations
   with Dini mean oscillation coefficients},
   journal={Indiana Univ. Math. J.},
   volume={69},
   date={2020},
   number={6},
   pages={1815--1853},
   issn={0022-2518},
   review={\MR{4170081}},
   doi={10.1512/iumj.2020.69.8028},
}


\bib{DonPha20}{article}{
   author={Dong, Hongjie},
   author={Phan, Tuoc},
   title={On parabolic and elliptic equations with singular or degenerate coefficients},
   pages={32pp},
   date={2020},
   status={	arXiv:2007.04385 preprint},
 }




\bib{FabGarMarSal88}{article}{
   author={Fabes, Eugene B.},
   author={Garofalo, Nicola},
   author={Mar\'{\i}n-Malave, Santiago},
   author={Salsa, Sandro},
   title={Fatou theorems for some nonlinear elliptic equations},
   journal={Rev. Mat. Iberoamericana},
   volume={4},
   date={1988},
   number={2},
   pages={227--251},
   issn={0213-2230},
   review={\MR{1028741}},
   doi={10.4171/RMI/73},
}




\bib{FabKenSer82}{article}{
   author={Fabes, Eugene B.},
   author={Kenig, Carlos E.},
   author={Serapioni, Raul P.},
   title={The local regularity of solutions of degenerate elliptic
   equations},
   journal={Comm. Partial Differential Equations},
   volume={7},
   date={1982},
   number={1},
   pages={77--116},
   issn={0360-5302},
   review={\MR{643158}},
   doi={10.1080/03605308208820218},
}




\bib{Fer98}{article}{
   author={Ferrari, Fausto},
   title={On boundary behavior of harmonic functions in H\"{o}lder domains},
   journal={J. Fourier Anal. Appl.},
   volume={4},
   date={1998},
   number={4-5},
   pages={447--461},
   issn={1069-5869},
   review={\MR{1658624}},
   doi={10.1007/BF02498219},
}





\bib{GilTru01}{book}{
   author={Gilbarg, David},
   author={Trudinger, Neil S.},
   title={Elliptic partial differential equations of second order},
   series={Classics in Mathematics},
   publisher={Springer-Verlag, Berlin},
   date={2001},
   pages={xiv+517},
   isbn={3-540-41160-7},
   review={\MR{1814364}},
}



\bib{HwaKim20}{article}{
   author={Hwang, Sukjung},
   author={Kim, Seick},
   title={Green's function for second order elliptic equations in non-divergence form},
   journal={Potential Anal.},
   volume={52},
   date={2020},
   number={1},
   pages={27--39},
   issn={0926-2601},
   review={\MR{4056785}},
   doi={10.1007/s11118-018-9729-z},
}




\bib{JerKen82}{article}{
   author={Jerison, David S.},
   author={Kenig, Carlos E.},
   title={Boundary behavior of harmonic functions in nontangentially
   accessible domains},
   journal={Adv. in Math.},
   volume={46},
   date={1982},
   number={1},
   pages={80--147},
   issn={0001-8708},
   review={\MR{676988}},
   doi={10.1016/0001-8708(82)90055-X},
}



\bib{Kem72}{article}{
   author={Kemper, John T.},
   title={A boundary Harnack principle for Lipschitz domains and the
   principle of positive singularities},
   journal={Comm. Pure Appl. Math.},
   volume={25},
   date={1972},
   pages={247--255},
   issn={0010-3640},
   review={\MR{293114}},
   doi={10.1002/cpa.3160250303},
}





\bib{Kuk22}{article}{
   author={Kukuljan, Teo},
   title={Higher order parabolic boundary Harnack inequality in $C^1$ and
   $C^{k,\alpha}$ domains},
   journal={Discrete Contin. Dyn. Syst.},
   volume={42},
   date={2022},
   number={6},
   pages={2667--2698},
   issn={1078-0947},
   review={\MR{4421508}},
   doi={10.3934/dcds.2021207},
}
	


\bib{LaMLeoSch20}{article}{
   author={La Manna, Domenico Angelo},
   author={Leone, Chiara},
   author={Schiattarella, Roberta},
   title={On the regularity of very weak solutions for linear elliptic
   equations in divergence form},
   journal={NoDEA Nonlinear Differential Equations Appl.},
   volume={27},
   date={2020},
   number={5},
   pages={Paper No. 43, 23},
   issn={1021-9722},
   review={\MR{4130239}},
   doi={10.1007/s00030-020-00646-8},
}



\bib{LinLin22}{article}{
   author={Lin, Fanghua},
   author={Lin, Zhengjiang},
   title={Boundary Harnack principle on nodal domains},
   journal={Sci. China Math.},
   volume={65},
   date={2022},
   number={12},
   pages={2441--2458},
   issn={1674-7283},
   review={\MR{4514975}},
   doi={10.1007/s11425-022-2016-3},
}


\bib{LogMal15}{article}{
   author={Logunov, Alexander},
   author={Malinnikova, Eugenia},
   title={On ratios of harmonic functions},
   journal={Adv. Math.},
   volume={274},
   date={2015},
   pages={241--262},
   issn={0001-8708},
   review={\MR{3318150}},
   doi={10.1016/j.aim.2015.01.009},
}


\bib{LogMal16}{article}{
   author={Logunov, Alexander},
   author={Malinnikova, Eugenia},
   title={Ratios of harmonic functions with the same zero set},
   journal={Geom. Funct. Anal.},
   volume={26},
   date={2016},
   number={3},
   pages={909--925},
   issn={1016-443X},
   review={\MR{3540456}},
   doi={10.1007/s00039-016-0369-4},
}






\bib{TorLat21}{article}{
   author={Ros-Oton, Xavier},
   author={Torres-Latorre, Dami\`a},
   title={New boundary Harnack inequalities with right hand side},
   journal={J. Differential Equations},
   volume={288},
   date={2021},
   pages={204--249},
   issn={0022-0396},
   review={\MR{4246154}},
   doi={10.1016/j.jde.2021.04.012},
}



\bib{SirTerVit21a}{article}{
   author={Sire, Yannick},
   author={Terracini, Susanna},
   author={Vita, Stefano},
   title={Liouville type theorems and regularity of solutions to degenerate
   or singular problems part I: even solutions},
   journal={Comm. Partial Differential Equations},
   volume={46},
   date={2021},
   number={2},
   pages={310--361},
   issn={0360-5302},
   review={\MR{4207950}},
   doi={10.1080/03605302.2020.1840586},
}
 
 
 

 \bib{SirTerVit21b}{article}{
   author={Sire, Yannick},
   author={Terracini, Susanna},
   author={Vita, Stefano},
   title={Liouville type theorems and regularity of solutions to degenerate
   or singular problems part II: odd solutions},
   journal={Math. Eng.},
   volume={3},
   date={2021},
   number={1},
   pages={Paper No. 5, 50},
   review={\MR{4144100}},
   doi={10.3934/mine.2021005},
}



\bib{TerTorVit22}{article}{
   author={Terracini, Susanna},
   author={Tortone, Giorgio},
   author={Vita, Stefano},
   title={Higher order boundary Harnack principle via degenerate equations},
   pages={35pp},
   date={2022},
   status={	arXiv:2301.00227 preprint},
 }
 
 
 
  \bib{Vit22}{article}{
   author={Vita, Stefano},
   title={Boundary regularity estimates in H\"{o}lder spaces with variable
   exponent},
   journal={Calc. Var. Partial Differential Equations},
   volume={61},
   date={2022},
   number={5},
   pages={Paper No. 166, 31},
   issn={0944-2669},
   review={\MR{4444179}},
   doi={10.1007/s00526-022-02274-9},
}
 
 
 
 \bib{Wu78}{article}{
   author={Wu, Jang Mei G.},
   title={Comparisons of kernel functions, boundary Harnack principle and
   relative Fatou theorem on Lipschitz domains},
   language={English, with French summary},
   journal={Ann. Inst. Fourier (Grenoble)},
   volume={28},
   date={1978},
   number={4},
   pages={147--167, vi},
   issn={0373-0956},
   review={\MR{513884}},
}




\bib{Zha23}{article}{
   author={Zhang, Chilin},
   title={On higher order boundary Harnack and analyticity of free boundaries},
   pages={20pp},
   date={2023},
   status={	arXiv:2303.04182 preprint},
 }


%\bib{RosTor21}{article}{
%   author={Ros-Oton, Xavier},
%   author={Torres-Latorre, Dami\`a},
%   title={New boundary {H}arnack inequalities with right hand side},
%   journal={J. Differential Equations},
%   volume={288},
%   date={204--249},
%   issn={0022-0396},
%   review={\MR{4246154}},
%   doi={10.1016/j.jde.2021.04.012},
%   }





 
\end{biblist}
\end{bibdiv}
\end{document}